\documentclass{article}

\usepackage[utf8]{inputenc}
\usepackage{fullpage}
\usepackage{amsmath,amssymb,amsthm,commath, subfig, float}
\usepackage[style=numeric]{biblatex}
\renewbibmacro{in:}{}

\usepackage{todonotes}
\usepackage{amsmath}

\title{On Structural Aspects of Friends-And-Strangers Graphs}
\author{Ryan Jeong}

\newcommand{\FS}{\mathsf{FS}}

\newcommand{\Path}{\text{Path}}
\newcommand{\Cycle}{\text{Cycle}}
\newcommand{\Grid}{\text{Grid}}

\newcommand{\Star}{\text{Star}}

\theoremstyle{definition}
\newtheorem{counter}{ctr}[section]
\newtheorem{theorem}[counter]{Theorem}
\newtheorem{definition}[counter]{Definition}
\newtheorem{conjecture}[counter]{Conjecture}
\newtheorem{corollary}[counter]{Corollary}
\newtheorem{proposition}[counter]{Proposition}
\newtheorem{lemma}[counter]{Lemma}
\newtheorem{remark}[counter]{Remark}

\newtheorem{problem}[counter]{Problem}
\newtheorem{example}[counter]{Example}

\begin{document}

\maketitle

\begin{abstract}
    Given two graphs $X$ and $Y$ with the same number of vertices, the friends-and-strangers graph $\FS(X, Y)$ has as its vertices all $n!$ bijections from $V(X)$ to $V(Y)$, with bijections $\sigma, \tau$ adjacent if and only if they differ on two elements of $V(X)$, whose mappings are adjacent in $Y$. In this article, we study necessary and sufficient conditions for $\FS(X, Y)$ to be connected for all graphs $X$ from some set. In the setting that we take $X$ to be drawn from the set of all biconnected graphs, we prove that $\FS(X, Y)$ is connected for all biconnected $X$ if and only if $\overline{Y}$ is a forest with trees of jointly coprime size; this resolves a conjecture of Defant and Kravitz. We also initiate and make significant progress toward determining the girth of $\FS(X, \Star_n)$ for connected graphs $X$, and in particular focus on the necessary trajectories that the central vertex of $\Star_n$ takes around all such graphs $X$ to achieve the girth.
\end{abstract}

\section{Introduction}

Defant and Kravitz (\cite{defant2020friends}) recently introduced friends-and-strangers graphs, which are defined as follows.

\begin{definition}[\cite{defant2020friends}] \label{fs_def}
Let $X$ and $Y$ be two simple graphs, each with $n$ vertices. The \textbf{friends-and-strangers} graph of $X$ and $Y$, denoted $\FS(X, Y)$, is a graph with vertices consisting of all bijections from $V(X)$ to $V(Y)$, with any two such bijections $\sigma, \sigma'$ adjacent in $\FS(X, Y)$ if and only if there exists an edge $\{a, b\}$ in $X$ such that the following hold.
\begin{itemize}
    \item $\{\sigma(a), \sigma(b)\} \in E(Y)$
    \item $\sigma(a) = \sigma'(b), \ \sigma(b) = \sigma'(a)$
    \item $\sigma(c) = \sigma'(c)$ for all $c \in V(X) \setminus \{a, b\}$.
\end{itemize}
In other words, $\sigma$ and $\sigma'$ differ precisely on two adjacent vertices of $X$, and the corresponding mappings are adjacent in $Y$. For any such $\sigma, \sigma'$, we say that $\sigma'$ is achieved from $\sigma$ by an \textbf{$(X, Y)$-friendly swap.}
\end{definition}

\begin{example}
See Figure \ref{fig:defn} for an illustration of this definition.
\end{example}

\begin{figure}[ht]

\begin{minipage}{.5\linewidth}
\centering
\subfloat[The graph $X$.]{\label{fig:X_gph}\includegraphics[width=0.42\textwidth]{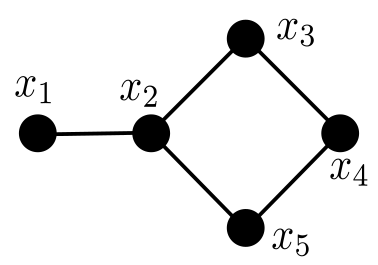}}
\end{minipage}%
\hfill
\begin{minipage}{.5\linewidth}
\centering
\subfloat[The graph $Y$.]{\label{fig:Y_gph}\includegraphics[width=0.42\textwidth]{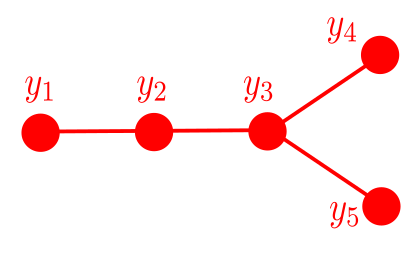}}
\end{minipage}\par\medskip
\centering
\subfloat[A sequence of $(X, Y)$-friendly swaps. The transpositions between adjacent configurations denote the two vertices in the graph $X$ involved in the $(X, Y)$-friendly swap. Red text corresponds to vertices in $Y$ placed upon vertices of $X$, which are labeled in black; this will be a convention throughout the rest of the work. The leftmost configuration corresponds to the bijection $\sigma \in V(\FS(X, Y))$ such that $\sigma(x_1) = y_1$, $\sigma(x_2) = y_5$, $\sigma(x_3) = y_3)$, $\sigma(x_4) = y_4$, and $\sigma(x_5) = y_2$; the other configurations analogously correspond to vertices in $\FS(X, Y)$.]{\label{fig:friendly_swaps}\includegraphics[width=.99\textwidth]{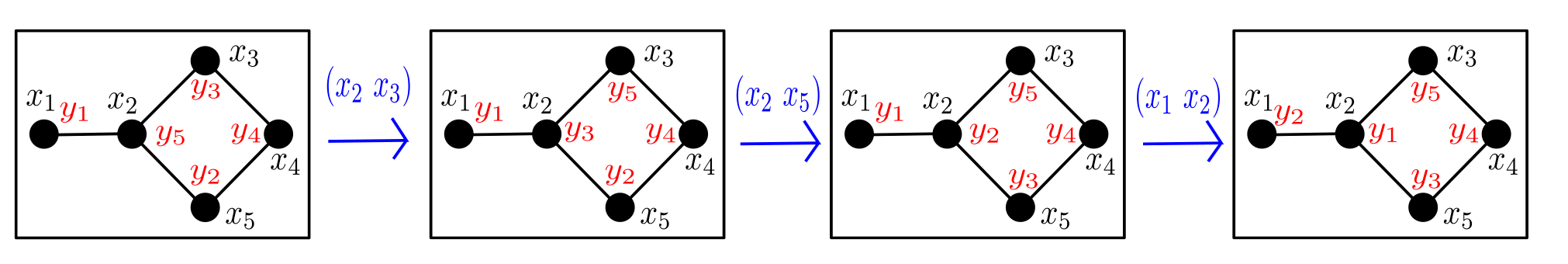}}

\caption{A sequence of $(X, Y)$-friendly swaps in $\FS(X, Y)$ for the graphs $X$ and $Y$, each on $5$ vertices. Any configuration in the bottom row corresponds to a vertex of $\FS(X, Y)$. Two consecutive configurations here differ by an $(X, Y)$-friendly swap, so the corresponding vertices in $\FS(X, Y)$ are adjacent.}
\label{fig:defn}
\end{figure}

\noindent As noted in \cite{defant2020friends}, it is frequently convenient to enumerate the vertices of the graphs $X$ and $Y$ so $V(X) = V(Y) = [n]$. Here, we can rephrase the definition as $V(\FS(X, Y)) = \mathfrak S_n$, and two permutations $\sigma, \sigma' \in V(\FS(X,Y))$ are adjacent if and only if
\begin{itemize}
    \item $\sigma' = \sigma \circ (i \ j)$ for some transposition $(i \ j)$
    \item $\{i, j\} \in E(X)$ 
    \item $\{\sigma(i), \sigma(j)\} \in E(Y)$.
\end{itemize}

\noindent The friends-and-strangers graph $\FS(X, Y)$ acquires its name from the following intuitive understanding of Definition \ref{fs_def}. Say that $V(X)$ corresponds to $n$ positions and $V(Y)$ corresponds to $n$ people, any two of whom are friends (if adjacent) or strangers (if nonadjacent). We place the $n$ people on the $n$ positions, with this configuration defining the bijection in $\FS(X, Y)$. From here, we can swap any two individuals if and only if their positions are adjacent in $X$ and the people placed on them are friends (i.e. adjacent in $Y$); this yields the bijection $\sigma' \in \FS(X, Y)$, for which we have $\{\sigma, \sigma'\} \in E(\FS(X, Y))$. 

A more concrete example of an object that friends-and-strangers graphs generalize is the famous $15$-puzzle, where the numbers $1$ through $15$ are placed on a $4$-by-$4$ board, with one empty space upon which adjacent tiles can slide. Indeed, if we let $X$ be the 4-by-4 grid graph $\Grid_{4 \times 4}$ and $Y = \Star_n$, then studying the graph $\FS(\Grid_{4 \times 4}, \Star_n)$ is equivalent to studying the set of possible configurations and moves that can be performed on the $15$-puzzle.
\subsection{Prior Work}

The article \cite{defant2020friends} introduces friends-and-strangers graphs and derives many of their basic properties. This work also studies the connected components of the graphs $\FS(\Path_n, Y)$ and $\FS(\Cycle_n, Y)$, as well as necessary and sufficient conditions for $\FS(X, Y)$ to be connected. We also remark that although friends-and-strangers graphs were introduced recently, many existing results in the literature can be recast into this framework. In particular, \cite{wilson1974graph} studies the connected components of $\FS(X, \Star_n)$ when $X$ is a biconnected graph. 

In a second paper by the same authors (and with Alon), \cite{alon2020typical} asks a number of probabilistic and extremal questions concerning friends-and-strangers graphs. The recent work \cite{bangachev2021asymmetric} provides asymmetric generalizations of two problems posed by \cite{alon2020typical}. Specifically, they study conditions on the minimal degrees of $X$ and $Y$ to guarantee that $\FS(X, Y)$ is connected, and a variant of this problem for $\FS(X, Y)$ to have two connected components when $X$ and $Y$ are taken to be edge-subgraphs of $K_{r, r}$, the complete bipartite graph with both partition classes having size $r$. 

Finally, in another paper \cite{jeong2021diameters}, we study the diameters of friends-and-strangers graphs, and in particular show that they fail to be polynomially bounded in the size of $X$ and $Y$. We also study the diameters of connected components of $\FS(\Path_n, Y)$ and $\FS(\Cycle_n, Y)$.

\subsection{Main Results}

In this paper, we explore two questions concerning structural properties of friends-and-strangers graphs.

\subsubsection{Connectivity of $\FS(X, Y)$ For Biconnected $X$}

Many results from \cite{alon2020typical,defant2020friends} strive to understand when $\FS(X, Y)$ is connected or disconnected, which corresponds in this setting to the ability to achieve any configuration from any other by an appropriate sequence of $(X, Y)$-friendly swaps. In one direction, we can ask for the set of graphs $Y$ such that $\FS(X, Y)$ is connected for all graphs $X$ from a specified set $\mathcal S$. Certainly all $X \in \mathcal S$ must be connected by Proposition \ref{basic_props} below; if we take $\mathcal S$ as the set of all connected graphs, then $Y = K_n$ is the only possibility. We now ask this question when $\mathcal S$ is the set of all biconnected graphs. The article \cite{defant2020friends} proves that $\FS(\Cycle_n, Y)$ is connected if and only if $Y$ is such that $\overline{Y}$ is a forest with trees of jointly coprime size, and conjectures that this is true for all biconnected graphs $X$. We devote the first part of this article towards resolve this conjecture, yielding the following result.

\begin{theorem}
Let $Y$ be a graph on $n \geq 3$ vertices such that $\overline{Y}$ is a forest consisting of trees $\mathcal T_1, \dots, \mathcal T_4$ with $\gcd(|V(\mathcal T_1)|, \dots, |V(\mathcal T_r)|) = 1$. If $X$ is a biconnected graph on $n$ vertices, then $\FS(X, Y)$ is connected.
\end{theorem}

\subsubsection{Girth}

In the latter part of this article, we explore the girth (the size of the smallest cycle subgraph) of a friends-and-strangers graph. In this context, this corresponds to the shortest sequence of $(X, Y)$-friendly swaps such that we begin and end in the same configuration, and two consecutive such swaps in this sequence do not involve the same vertices in $X$. This can easily be reduced to studying $\FS(X, \Star_n)$ for connected graphs $X$, for which we pose the following problem.

\begin{problem}
Find a precise description of the set of simple graphs $\mathcal G$ with finite girth such that for all $n \geq 3$, the following two statements hold.
\begin{itemize}
    \item Any $n$-vertex connected graph $X$ with $g(X) < \infty$ has that $g(\FS(X, \Star_n)) = g(\FS(\Tilde{X}, \Star_m))$ for some subgraph $\Tilde{X}$ of $X$ that is in $\mathcal G$ on $m \leq n$ vertices.
    \item If $X \in \mathcal G$, the only subgraph $\tilde{X} \in \mathcal G$ of $X$ satisfying $g(\FS(X, \Star_n)) = g(\FS(\Tilde{X}, \Star_m))$ is $\tilde{X} = X$ itself.
\end{itemize}
\end{problem}
\noindent Fundamentally, this problem (which we show is well-defined) asks for the necessary trajectories that the central vertex of $\Star_n$ must take around a graph $X$ to achieve the girth of $\FS(X, \Star_n)$. We make substantial progress in a complete characterization of the set of simple graphs $\mathcal G$, and in particular have the following results. In particular, the main body of the work provides specific trajectories that the central vertex of $\Star_n$ traverses around the graphs $\tilde{\mathcal G}$ described below to achieve a cycle in $\FS(X, \Star_n)$.

\begin{theorem} \label{girth_subset_intro}
Let $\tilde{\mathcal G}$ be the set of simple graphs that include the following.
\small
\begin{itemize}
    \item All cycle graphs.
    \item Barbell graphs $[\mathcal C_1, \mathcal C_2, \mathcal P]$ with $2(|V(\mathcal C_1)| + |V(\mathcal C_2)| + 2|E(\mathcal P)|) < \min\{|V(\mathcal C_1)|(|V(\mathcal C_1)|-1), |V(\mathcal C_2)|(|V(\mathcal C_2)|-1)\}$.
    \item $\theta$-graphs with $2(|V(\mathcal C_1)| + |V(\mathcal C_2)| + |V(\mathcal C)|) < \min\{|V(\mathcal C_1)|(|V(\mathcal C_1)|-1), |V(\mathcal C_2)|(|V(\mathcal C_2)|-1), |V(\mathcal C)|(|V(\mathcal C)|-1)\}$.
    \item $\tilde{\theta}$-graphs with $3(|V(\mathcal C_1)| + |V(\mathcal C_2)| + |V(\mathcal C)|) < \min\{|V(\mathcal C_1)|(|V(\mathcal C_1)|-1), |V(\mathcal C_2)|(|V(\mathcal C_2)|-1), |V(\mathcal C)|(|V(\mathcal C)|-1)\}$.
    \item $\tilde{\theta_4}$-graphs such that $4+4p_2+2(p_3+p_4) < \min\{p_2(p_2+1), 4(p_2+p_3+p_4), 6(1+p_2+p_3)\}$, where $p_1 = 1 < p_2 \leq p_3 \leq p_4$ denote the lengths of the paths between the two vertices of degree $4$.
\end{itemize}
\normalsize
\noindent We have that $\tilde{\mathcal G} \subset \mathcal G$.
\end{theorem}

\begin{figure}[ht]
    \centering
    \begin{minipage}{.24\linewidth}
    \centering
    \subfloat[Barbell graphs.]{\label{fig:barbell_intro}\includegraphics[width=\textwidth]{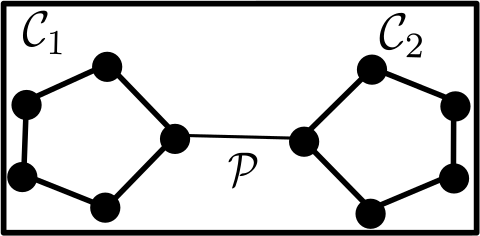}}
    \end{minipage}%
    \hfill
    \begin{minipage}{.24\linewidth}
    \centering
    \subfloat[$\theta$-graphs.]{\label{fig:theta_intro}\includegraphics[width=\textwidth]{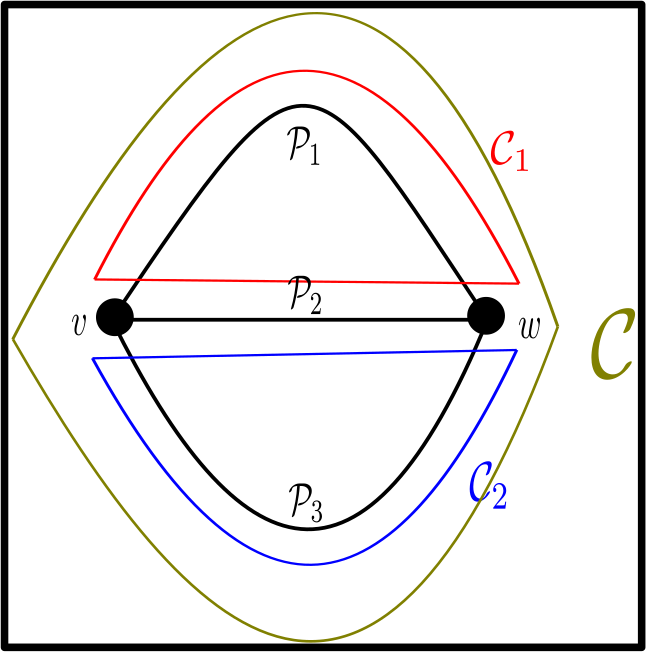}}
    \end{minipage}%
    \hfill
    \begin{minipage}{.24\linewidth}
    \centering
    \subfloat[$\tilde{\theta}$-graphs.]{\label{fig:tilde_theta_intro}\includegraphics[width=\textwidth]{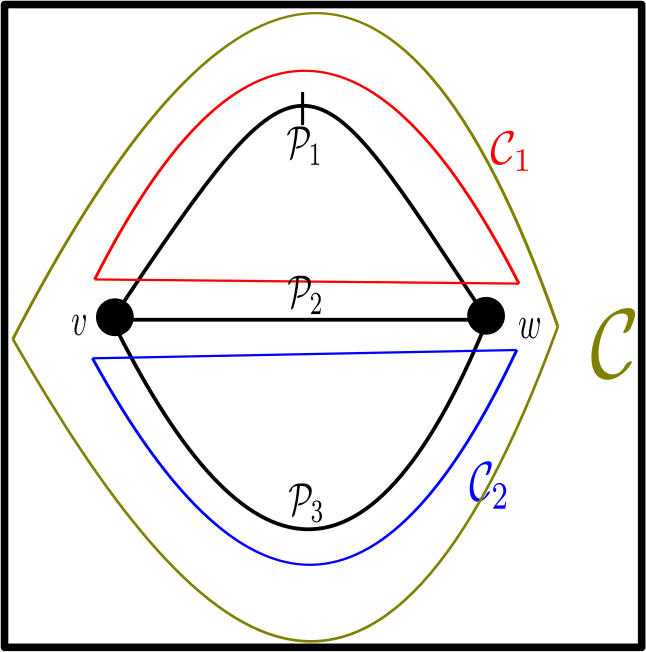}}
    \end{minipage}
    \hfill
    \begin{minipage}{.24\linewidth}
    \centering
    \subfloat[$\tilde{\theta_4}$-graphs.]{\label{fig:tilde_theta4_intro}\includegraphics[width=\textwidth]{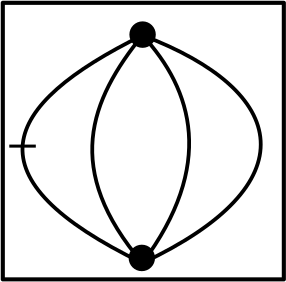}}
    \end{minipage}%
    \caption{Graphs in $\tilde{\mathcal G}$, as discussed in Theorem \ref{girth_subset_intro}. Hatch marks over paths indicate that they are edges.}
    \label{fig:girth_subset_intro}
\end{figure}

\noindent See Figure \ref{girth_subset_intro} for illustrations of the graphs in Theorem \ref{girth_subset_intro}. We conjecture that $\mathcal G = \tilde{\mathcal G}$, and will leave this unresolved in this work. Instead, we provide the following superset of $\mathcal G$, which shows that not much can be added to $\tilde{\mathcal G}$ to achieve $\mathcal G$.

\begin{theorem} \label{girth_superset_intro}
Let $\mathcal G'$ include the set $\tilde{\mathcal G}$ from Theorem \ref{girth_subset_intro}, as well as all instances of the following graphs. Here, all ``edges" in Figure \ref{fig:girth_superset_intro} correspond to paths. 
\small
\begin{itemize}
    \item $\theta_4$-graphs.
    \item $\theta_5$-graphs.
    \item $K_4^*$-graphs.
    \item $K_{3, 3}^*$-graphs.
\end{itemize}
\normalsize
\noindent We have that $\mathcal G \subset \mathcal G'$.
\end{theorem}

\begin{figure}[ht]
    \centering
    \begin{minipage}{.195\linewidth}
    \centering
    \subfloat[$\theta_4$-graphs.]{\label{fig:theta4_intro}\includegraphics[width=\textwidth]{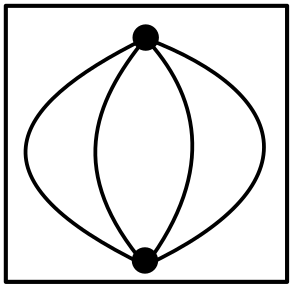}}
    \end{minipage}%
    \hfill
    \begin{minipage}{.195\linewidth}
    \centering
    \subfloat[$\theta_5$-graphs.]{\label{fig:theta5_intro}\includegraphics[width=\textwidth]{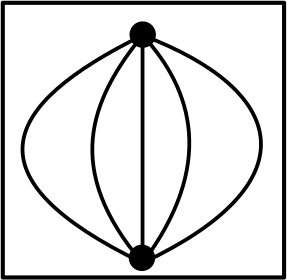}}
    \end{minipage}%
    \hfill
    \begin{minipage}{.25\linewidth}
    \centering
    \subfloat[$K_4^*$-graphs.]{\label{fig:K4*_intro}\includegraphics[width=\textwidth]{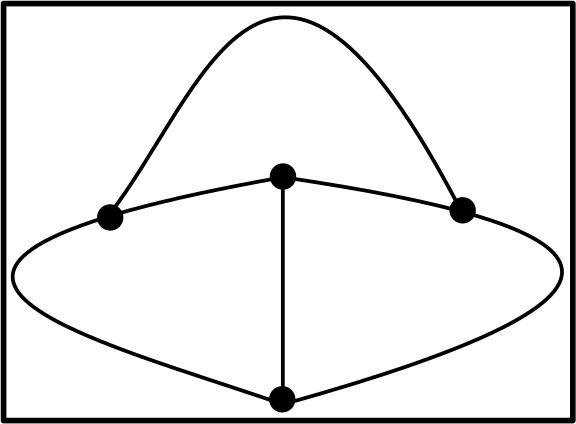}}
    \end{minipage}
    \hfill
    \begin{minipage}{.25\linewidth}
    \centering
    \subfloat[$K_{3,3}^*$-graphs.]{\label{fig:K33*_intro}\includegraphics[width=\textwidth]{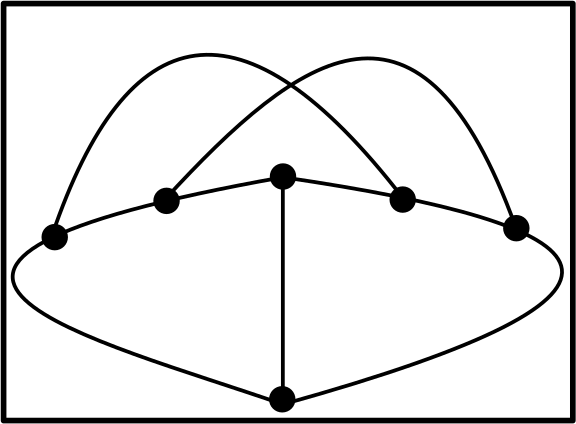}}
    \end{minipage}%
    \caption{Graphs in $\mathcal G'$, as discussed in Theorem \ref{girth_superset_intro}.}
    \label{fig:girth_superset_intro}
\end{figure}

\section{Preliminaries}

\subsection{Notation}
Here, we review some common families of graphs and elementary graph theory terminology that we shall refer to throughout this article. 

\begin{itemize}
    \item $[n] = \{1, 2, \dots, n\}$. 
    \item The vertex and edge sets of a graph $G$ will be denoted $V(G)$ and $E(G)$, respectively.
    \item Define the disjoint union of a collection of graphs $\{G_i\}_{i \in I}$, notated $\bigoplus_{i \in I} G_i$, to be the graph with vertex set $\bigsqcup_{i \in I} V(G_i)$ and edge set $\bigsqcup_{i \in I} E(G_i)$. This readily extends to expressing a graph as the disjoint union of its connected components. 
    \item The Cartesian product of graphs $G_1, \dots, G_r$, denoted $G_1 \square \dots \square G_r$, has vertex set $V(G_1) \times \dots \times V(G_r)$, with $(v_1, \dots, v_r)$ and $(w_1, \dots, w_r)$ adjacent if and only if there exists $i \in [r]$ such that $\{v_i, w_i\} \in E(G_i)$ and $v_j = w_j$ for all $j \in [r] \setminus \{i\}$.
\end{itemize}

\subsubsection{Common Families of Graphs} \label{graph_families}

Assume that the vertex set of all graphs is given by $[n]$. We define the graphs in terms of their edge sets.
\begin{itemize}
    \item The complete graph $K_n$ has edge set $\{\{i, j\} : i, j \in [n], i \neq j\}$.
    \item The path graph $\Path_n$ has edge set $E(\Path_n) = \{\{i, i+1\} : i \in [n-1]\}$.
    \item The cycle graph $\Cycle_n$ has edge set $E(\Cycle_n) = \{\{i, i+1\} : i \in [n-1]\} \cup \{\{n, 1\}\}$.
    \item The star graph $\Star_n$ has edge set $E(\Star_n) = \{\{i, n\} : i \in [n-1]\}$.
    \item For $i + j = n$, the complete bipartite graph $K_{i,j}$ has edge set $E(K_{i,j}) = \{\{v_1, v_2\}: v_1 \in [i], v_2 \in \{i+1, n\}\}$. This partitions $V(K_{i,j})$ into two sets so that every vertex in one set is adjacent to every vertex in the other; we shall refer to these sets as partition classes of $V(K_{i, j})$.
\end{itemize}

\subsubsection{Relevant Terminology}

We review the definitions for some elementary notions in graph theory.
\begin{itemize}
    \item The \textit{complement} of a graph $G$, denoted $\overline{G}$, is the graph with vertex set $V(G)$, such that for any $v, w \in V(G)$ with $v \neq w$, we have that $\{v, w\} \in E(\overline{G})$ if and only if $\{v, w\} \notin E(G)$. 
    \item An \textit{isomorphism} from $G$ to $H$ is a mapping $\varphi: V(G) \to V(H)$ such that $\{v, w\} \in E(G)$ if and only if $\{\varphi(v), \varphi(w)\} \in E(H)$. Here, we say that $G$ and $H$ are \textit{isomorphic}, and denote the fact that two graphs are isomorphic by $G \cong H$.
    \item A graph $H$ is a \textit{subgraph} of a graph $G$ if $V(H) \subseteq V(G)$ and $E(H) \subseteq E(G)$. The subgraph $H$ is said to be induced if $E(H) = \{\{v, w\} \in E(G) : v, w \in V(H)\}$. In particular, if the vertex set of $H$ is given by the set $\mathcal V = V(H)$, we shall denote the induced subgraph $H$ by $G|_{\mathcal V}$.
    \item A graph $G$ is \textit{connected} if and only if for any $v, w \in V(G)$, there exists a path in $G$ connecting $v$ to $w$. A \textit{connected component} of $G$ is a maximal connected subgraph of $G$; we say that the size of a connected component $H$ of $G$ is $|V(H)|$. In particular, if the connected components of $G$ are given by $H_1, \dots, H_r$, then we can write $G = \bigoplus_{i=1}^r H_i$ as the disjoint union of its connected components.
    \item A \textit{cut vertex} of a graph $G$ is a vertex $v \in V(G)$ such that removing the vertex and all incident edges causes the resulting graph (which is precisely given by $G |_{V(G) \setminus v}$) to be disconnected. A graph $G$ is \textit{biconnected} (sometimes called 2-connected) if it is connected and does not have a cut vertex. 
    \item The \textit{distance} $d_G(x,y)$ in $G$ of two vertices $v, w \in V(G)$ is the length of a shortest path between $x$ and $y$ in $G$. We shall drop the subscript $G$ and write $d(v, w)$ if the graph $G$ is obvious from context, and say $d(v,w) = \infty$ if no such path exists. The \textit{diameter} of $G$ is the greatest distance between any two vertices in $V(G)$; if $G$ is not connected, we shall often study diameters of connected components of $G$.
    \item The \textit{girth} of a graph $G$, denoted $g(G)$, is the size of the smallest cycle subgraph contained in $G$. We say that $g(G) = \infty$ if $G$ does not have any cycle subgraphs. A graph $G$ on $n$ vertices is \textit{Hamiltonian} if it contains $\Cycle_n$ as a subgraph.
\end{itemize}

\subsection{Basic Properties}

It is straightforward to observe that the isomorphism type of $\FS(X, Y)$ depends strictly on the isomorphism types of the graphs $X$ and $Y$. The following proposition gives other elementary properties of $\FS(X,Y)$.
\begin{proposition}[\cite{defant2020friends}] \label{basic_props}
Let $\FS(X, Y)$ be the friends-and-strangers graph of $X$ and $Y$. 
\begin{enumerate}
    \item $\FS(X,Y)$ is isomorphic to $\FS(Y, X)$.
    \item $\FS(X,Y)$ is bipartite.
    \item If $X$ or $Y$ is disconnected, then $\FS(X, Y)$ is also disconnected.
    \item Let $X, \Tilde{X}, Y, \Tilde{Y}$ be graphs on $n$ vertices. If $X$ is isomorphic to a subgraph of $\Tilde{X}$ and $Y$ is isomorphic to a subgraph of $\Tilde{Y}$, then $\FS(X,Y)$ is isomorphic to a subgraph of $\FS(\Tilde{X}, \Tilde{Y})$. 
    \item Let $X$ and $Y$ be connected graphs on $n \geq 3$ vertices, each with a cut vertex. Then $\FS(X, Y)$ is disconnected.
\end{enumerate}
\end{proposition}

\noindent We generalize Property (4) of Proposition \ref{basic_props}, as we shall need it later; this can be proved by a straightforward generalization of the argument provided in \cite{defant2020friends}. 
\begin{proposition} \label{subgraph_lemma}
Let $X, Y$ be graphs on $m$ vertices and $\Tilde{X}, \Tilde{Y}$ be graphs on $n \geq m$ vertices. If $X$ is isomorphic to a subgraph of $\Tilde{X}$ and $Y$ is isomorphic to a subgraph of $\Tilde{Y}$, then $\FS(X,Y)$ is isomorphic to a subgraph of $\FS(\Tilde{X}, \Tilde{Y})$. In particular, for every bijection $\psi: V(\Tilde{X}) \setminus V(X) \to V(\Tilde{Y}) \setminus V(Y)$ (considering the isomorphic copies of $X$ and $Y$ in $\Tilde{X}$ and $\Tilde{Y}$, respectively), there exists a subgraph of $\FS(\Tilde{X}, \Tilde{Y})$ isomorphic to $\FS(X, Y)$ which has as its vertices all bijections consistent with $\psi$.
\end{proposition}

\noindent The following proposition characterizes $\FS(X, Y)$ in terms of the components of $X$, and in particular shows that we need only consider the setting in which the graphs $X$ and $Y$ are connected.
\begin{proposition}[\cite{defant2020friends}] \label{fs_decomp}
For graphs $X$ and $Y$, let $X_1, \dots, X_r$ be the connected components of $X$, with cardinalities $n_1, \dots, n_r$, respectively. Let $\mathcal{OP}_{n_1, \dots, n_r}(Y)$ denote the collection of ordered set partitions $(V_1, \dots, V_r)$ of $V(Y)$ such that $|V_i| = n_i$ for all $i \in [r]$. Then 
\begin{align*}
    \FS(X, Y) \cong \bigoplus_{(V_1, \dots, V_r) \in \mathcal{OP}_{n_1, \dots, n_r}(Y)} (\FS(X_1, Y|_{V_1}) \square \dots \square \FS(X_r, Y |_{V_r}))
\end{align*}
\end{proposition}

\section{On Making $\FS(X, Y)$ Connected For All ``Reasonable" $X$}

As remarked in \cite{defant2020friends}, one direction of study is to take a ``reasonable" class of all graphs $X$ on $n$ vertices satisfying some property, and determine the sparsest graph $Y$ on $n$ vertices such that $\FS(X, Y)$ is connected. Certainly, our set should be restricted to some subset of connected graphs, as $\FS(X, Y)$ would be disconnected for some choice of $X$ otherwise. If we consider the set of all connected $X$, then necessarily $Y = K_n$ since $\FS(\Path_n, Y)$ is connected if and only if $Y = K_n$ (see Theorem 3.1 of \cite{defant2020friends}). We can similarly ask this question when we take our set to be all biconnected graphs $X$. The simplest example of an $n$-vertex biconnected graph is considered to be $\Cycle_n$, for which there are known results concerning the connectivity of $\FS(\Cycle_n, Y)$.

\subsection{Background}

Corollary 4.14 of \cite{defant2020friends} states the following, which characterizes when $\FS(\Cycle_n, Y)$ is connected.
\begin{theorem}[\cite{defant2020friends}] \label{orig_result}
For $n \geq 3$, the graph $\FS(\Cycle_n, Y)$ is connected if and only if $\overline{Y}$ is a forest consisting of trees $\mathcal T_1, \dots, \mathcal T_r$ such that $\gcd(|V(\mathcal T_1)|, \dots, |V(\mathcal T_r)|) = 1$.
\end{theorem}

\noindent By Theorem \ref{orig_result}, we must have that $Y$ is such that $\overline{Y}$ is a forest with trees of jointly coprime size, since $\Cycle_n$ is a biconnected graph; as such, $Y$ is necessarily of this form. Conjecture 7.1 of the same paper (Theorem \ref{conjec_7.1} below) aims to show that this is also sufficient, namely that for any biconnected $X$, any $Y$ satisfying this statement yields that $\FS(X, Y)$ is connected.

\begin{theorem} \label{conjec_7.1}
Let $Y$ be a graph on $n \geq 3$ vertices, such that $\overline{Y}$ is a forest consisting of trees $\mathcal T_1, \dots, \mathcal T_r$ with $\gcd(|V(\mathcal T_1)|, \dots, |V(\mathcal T_r)|) = 1$. If $X$ is a biconnected graph on $n$ vertices, then $\FS(X,Y)$ is connected.
\end{theorem}
\noindent This section is dedicated to proving this statement. Toward this, we begin by providing important results invoked throughout the proceeding argument. 

\begin{definition} \label{open_ear_decomp}
An \textbf{open ear decomposition} of a graph $G$ is a finite ordered sequence of subgraphs $[P_0, P_1, \dots, P_r]$ of $G$ such that the following properties hold.
\begin{itemize}
    \item The edge sets of the $P_i$ partition the edges of $G$, i.e.
    $E(G) = \bigsqcup_{i=0}^r E(P_i)$.
    \item $P_0$ is a simple cycle, and for all $i \geq 1$, $P_i$ is a path that is not a simple cycle (i.e. its endpoints are not the same vertex in the graph $G$).
    \item For $i \geq 1$, each endpoint of $P_i$ is contained in some previous subgraph $P_j$, for $j < i$, while all internal vertices of $P_i$ are not contained in any such $P_j$.
\end{itemize}
\end{definition}

\begin{figure}[ht]
    \centering
    \includegraphics[width=0.2\textwidth]{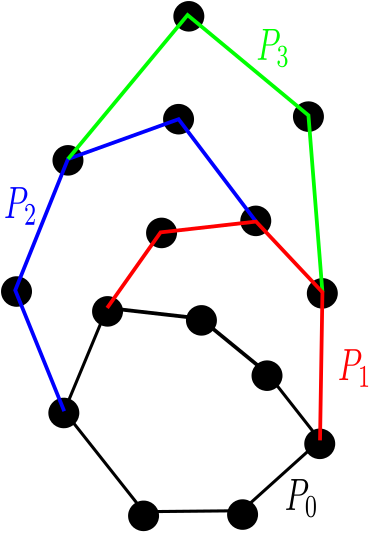}
    \caption{The open ear decomposition of a biconnected graph, illustrated here for $r=3$.}
    \label{fig:open_ear_fig}
\end{figure}

\noindent This yields a natural characterization of all biconnected graphs on $n \geq 3$ vertices. 
\begin{proposition}[\cite{whitney1992non}] \label{bicon}
A simple graph $G$ on $n \geq 3$ vertices has an open ear decomposition if and only if it is biconnected.
\end{proposition}

\noindent The following proposition will be referenced briefly in our study of biconnected graphs with one ear. In particular, it follows immediately from this proposition that any open ear decomposition of a biconnected graph $X$ must have the same number of ears $r$.
\begin{proposition}[\cite{miller1986efficient}] \label{ear_eq}
A biconnected graph $G$ with an open ear decomposition with $r$ ears must satisfy $r = |E(G)| - |V(G)|$. ($r$ is sometimes called the Betti number of $G$.)
\end{proposition}

\noindent Define the graph $\theta_0$ to be the following, relevant for the case where $n = 7$.

\begin{figure}[ht]
    \centering
    \includegraphics[width=0.15\textwidth]{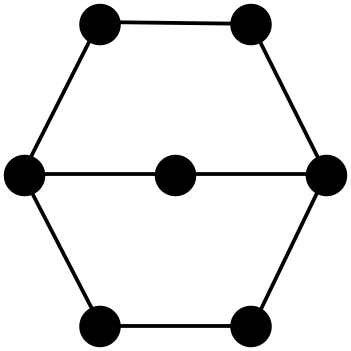}
    \caption{The graph $\theta_0$.}
    \label{fig:theta_0}
\end{figure}

\begin{proposition}[\cite{wilson1974graph}]
Let $Y$ be a biconnected graph on $n \geq 3$ vertices that is not isomorphic to $\theta_0$ or $\Cycle_n$. If $Y$ is not bipartite, then $\FS(\Star_n, Y)$ is connected. If $Y$ is bipartite, then $\FS(\Star_n, Y)$ has exactly two connected components, each of size $n!/2$. The graph $\FS(\Star_7, \theta_0)$ has exactly $6$ connected components.
\end{proposition}

\noindent The following corollary demonstrates a sense in which the preceding proposition is sharp. 
\begin{corollary}[\cite{defant2020friends}] \label{strong_wilson} Let $X$ be a graph on $n \geq 3$ vertices that contains $\Star_n$ as a proper subgraph, and $Y$ be a biconnected graph on $n$ vertices that is not isomorphic to $\Cycle_n$ or $\theta_0$. Then $\FS(X, Y)$ is connected.
\end{corollary}

\noindent We can now prove the following special case of Theorem \ref{conjec_7.1}, in which $\overline{Y}$ has an isolated vertex.
\begin{proposition} \label{isol_vx}
Let $Y$ be a graph on $n \geq 3$ vertices such that $\overline{Y}$ is a forest with an isolated vertex (i.e. a vertex of degree zero). If $X$ is a biconnected graph on $n$ vertices, then $\FS(X,Y)$ is connected.
\end{proposition}
\begin{proof}
The proposition follows directly from Theorem \ref{orig_result} in the case that $X = \Cycle_n$, since we have that $\gcd(|V(\mathcal T_1)|, \dots, |V(\mathcal T_r)|) = 1$ necessarily follows from the isolated vertex in $V(\overline{Y})$ (henceforth denoted $v$).

We now consider all other biconnected graphs $X$. First, assume that $X$ is not isomorphic to $\theta_0$. Here, the vertex $v$ yields that $Y$ has a subgraph isomorphic to $\Star_n$ (with center $v$), which is necessarily proper in $Y$ for all $n \geq 4$, for which the result follows immediately from Corollary \ref{strong_wilson}. If $n = 3$, then we must have $X = K_3$, and it is easy to directly confirm that $\FS(X, Y)$ is connected.

Now consider $X$ isomorphic to $\theta_0$ (for $n = 7$). The article \cite{defant2020friends} confirms by a computer check that $\FS(X, Y)$ is connected for any graph $Y$ such that $\overline{Y}$ consists of the disjoint union of an isolated vertex and a tree with $6$ vertices. Any $\overline{Y}$ satisfying the constraints in the proposition is a subgraph of such a graph (call it $\overline{Y'}$), so that $Y'$ is a subgraph of $Y$. Since $\FS(\theta_0, Y')$ is connected and a subgraph of $\FS(\theta_0, Y)$ by Proposition \ref{subgraph_lemma} with the same vertex set, it follows that $\FS(X, Y)$ is connected.
\end{proof}
\noindent Henceforth, we shall concern ourselves strictly with the setting in which $\overline{Y}$ has no isolated vertex. In this case, it follows quickly that $\overline{Y}$ must have at least four leaves, since the graph $\overline{Y}$ must consist of at least two trees, both of which necessarily have at least two vertices (and thus at least two leaves).

\subsection{Biconnected Graphs with One Ear} \label{for_r=1}
We begin by extending Theorem \ref{orig_result}
for all biconnected graphs that can be decomposed into an open ear decomposition with at most one ear (i.e. $r \leq 1$ in Definition \ref{open_ear_decomp}). We concern ourselves with the $r=1$ case ($r=0$ is the statement of Theorem \ref{orig_result}). The result is immediate for $n=3$ (any biconnected graph is necessarily $K_3$) and for $n=4$ (any biconnected graph with at most one ear is Hamiltonian), so take $n \geq 5$.

Let $[P_0, \ P_1]$ denote the open ear decomposition of some biconnected $G$ with one ear. Denote the vertices in the simple path $P_1$ by $\{v_0, v_1, \dots, v_m\}$ (as ordered in the path). Let $\{v_0', \dots, v'_k\}$ with $v_0' = v_0$, $v_k' = v_m$ be the shorter path from $v_0$ to $v_m$ in $P_0$, and $\{w_0, \dots, w_\ell\}$ with $w_0 = v_m$, $w_\ell = v_0$ the longer. We can assume $k, m \geq 2$ (i.e., there exist inner vertices in $P_1$ and the shorter path in $P_0$), as the connectedness of $\FS(X, Y)$ is immediate from $G$ being Hamiltonian if these bounds are not satisfied. We must have the strict inequality $k + m < n$, since $k < |E(P_0)| - 1$, $m = |E(P_1)|$, and by Proposition $\ref{ear_eq}$, $|E(P_0)| + |E(P_1)| = |E(X)| = n+1$. We can also assume, without loss of generality, that $k, m \leq \lfloor \frac{n}{2} \rfloor$. This is obvious for $k$ (the length of the shorter of two paths between the same vertices in $P_0$, and $|E(P_0)| < n$). If this were not true for $m$ (i.e. $m > \lfloor \frac{n}{2} \rfloor$), construct a different open ear decomposition $[P_0', P_1']$ of $G$ that includes the vertices in $P_1$ in the new initial cycle $P_0'$; by Proposition \ref{ear_eq}, $|E(P_0')| + |E(P_1')| = |E(X)| = n+1$ and $|E(P_0')| > \lfloor \frac{n}{2} \rfloor + 1$, so that the length of $P_1'$, $|E(P_1')|$, has the desired upper bound. 

With these reductions, construct the graph $X$ from $\Cycle_n$ by removing the edge $\{1, n\}$ from $E(\Cycle_n)$ and adding the edges $\{n, k\}$ and $\{1, n-m+1\}$, where $2 \leq k \leq \lfloor \frac{n}{2} \rfloor < n-m+1 \leq n-1$ and $k + m < n$. Note that $(n-m+1) - k = n - (k+m) + 1 \geq 2$, or there must be at least one vertex strictly between the vertices $k$ and $n-m+1$ in $\Cycle_n$. We now have the following observation.
\begin{proposition} \label{ear_fold}
Any non-Hamiltonian biconnected graph $G$ on $n \geq 5$ vertices with one ear is isomorphic to a graph $X$ with $V(X) = [n]$ and $E(X) = \{\{i, i+1\}: 1 \leq i \leq n-1\} \cup \{1,w\} \cup \{n, v\}$, with the inequalities $2 \leq v \leq \lfloor \frac{n}{2} \rfloor < w \leq n-1$ and $v + [(n+1)-w] < n$.
\end{proposition}

\begin{proof}
Construct $X$ from $G$ as in the preceding discussion: here, we have the correspondences $v = k$ and $w = n-m+1$, and $k + m + \ell = n + 1$ (by Proposition \ref{ear_eq}). Define the mapping $\varphi: [n] = V(X) \to V(G)$ by
\begin{align*}
    \varphi(v) = \begin{cases}
    v'_i & v = i, \ \ 1 \leq i \leq k \\
    w_i & v = k + i, \ \ 1 \leq i \leq \ell - 1 \\
    v_{m-i} & v = n-i+1, \ \ m \geq i \geq 1
    \end{cases}
\end{align*}
One can confirm that $\varphi$ is indeed a graph isomorphism, so $X \cong G$.
\end{proof}

\begin{figure}[ht]
    \centering
    \includegraphics[width=0.5\textwidth]{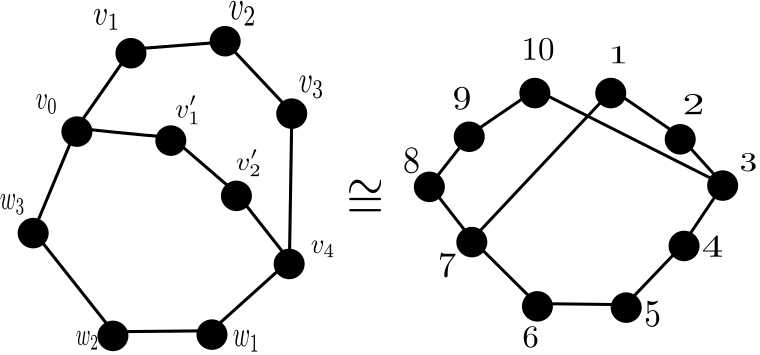}
    \caption{An illustration of the isomorphism described in Proposition \ref{ear_fold}.}
    \label{fig:r=1_decomp}
\end{figure}

\subsubsection{Structure of $X$ + Proof Overview}

Take a graph on $n \geq 5$ vertices as detailed in Proposition \ref{ear_fold}; we establish some notation for what follows. Partition $V(X) = [n]$ into three sets: $\mathcal B_1 = \{w, w+1, \dots, n\}, \mathcal B_2 = \{1, 2, \dots, v\}$, and $\mathcal B = \{v+1, v+2, \dots, w-1\}$. From the bounds on $v$ and $w$ above, $\mathcal B$ is nonempty, since $|\mathcal B| = w - v - 1 \geq 1$. There exist three subgraphs in $X$ that are isomorphic to a cycle graph. The cycle $\mathcal C$ goes across the path corresponding to $\mathcal B_1$, edge $\{1, w\}$, the path corresponding to $\mathcal B_2$, and edge $\{n, v\}$ (this is indeed a valid cycle graph due to the inequalities on $v$ and $w$ established above). The cycle $\mathcal C_1$ goes across $\{1, w\}$ and loops around $\mathcal B$ and $\mathcal B_2$. The cycle $\mathcal C_2$ goes across $\{n, v\}$ and loops around $\mathcal B$ and $\mathcal B_1$.

\begin{figure}[ht]
    \centering
    \includegraphics[width=0.3\textwidth]{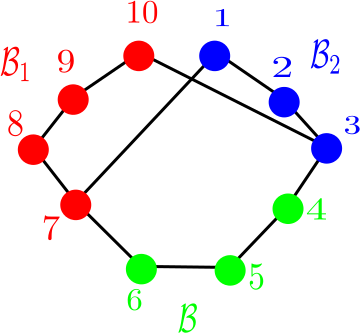}
    \caption{An illustration of the notation introduced for the example given in Figure \ref{fig:r=1_decomp}.}
    \label{fig:graph_labels}
\end{figure}

We motivate the core of the proof of Lemma \ref{r=1}. By Theorem \ref{orig_result}, $\FS(\Cycle_n, Y)$ is connected whenever $\overline{Y}$ is a forest with an isolated vertex. In particular, for some $\sigma: V(X) \to V(Y)$, consider the cycle subgraph $\mathcal C$ of $X$ and the subgraph $Y |_{\sigma(V(\mathcal C))}$ of $Y$ induced by all vertices on the positions of $\mathcal C$, or $\sigma(V(\mathcal C)) = \{\sigma(v) : v \in V(\mathcal C)\}$ in $Y$. Consider the subgraph $\FS(\mathcal C, Y |_{\sigma(V(\mathcal C))})$ of $\FS(X, Y)$. (More specifically, we refer to the subgraph in $\FS(X, Y)$ isomorphic to $\FS(\mathcal C, Y|_{\sigma(V(\mathcal C))})$ with all mappings on $\mathcal B$ consistent with $\sigma$; this clarification will be omitted in future references to this subgraph.) Since $\overline{Y}$ is a forest, the subgraph $\overline{Y |_{\sigma(V(\mathcal C))}}$ is also a forest. Thus, if there is an isolated vertex in $\overline{Y |_{\sigma(V(\mathcal C))}}$, then we can achieve any bijection of the values in $\sigma(V(\mathcal C))$ on the positions of $V(\mathcal C)$ via $(\mathcal C,Y |_{\sigma(V(\mathcal C))})$-friendly swaps by Theorem \ref{orig_result}.

Say we begin with some permutation $\sigma_0: V(X) \to V(Y)$. We perform a sequence of $(X,Y)$-friendly swaps yielding a permutation $\sigma: V(X) \to V(Y)$ such that there exists an isolated vertex in $\overline{Y |_{\sigma(V(\mathcal C))}}$, swap the positions of $\sigma(1)$ and $\sigma(n)$ by appealing to the preceding result, and return to the original configuration $\sigma_0$ on all vertices of $X$ excluding $1$ and $n$ (which have been interchanged). Specifically, for some $\ell, \ell' \in V(X)$, say that we have $\sigma(\ell)$ a leaf vertex in $\overline{Y}$, and $\sigma(\ell')$ the unique vertex in $\overline{Y}$ that $\sigma(\ell)$ is adjacent to. If $\ell \in V(\mathcal C)$ and $\ell' \notin V(\mathcal C)$, we can refer to Theorem \ref{orig_result} on the subgraph $\FS(\mathcal C, Y |_{\sigma(V(\mathcal C))})$ of $\FS(X, Y)$ to perform the desired swap. 

\subsubsection{$r=1$}
We prove Theorem \ref{conjec_7.1} for the case in which $r=1$. We begin with the following preliminary results.

\begin{remark} \label{trivial_leaf_prop}
Any tree $\mathcal T$ with $n \geq 2$ vertices has at least two leaves (vertices of degree $1$).
\end{remark}

\begin{proposition} \label{leaf_prop}
Any tree $\mathcal T$ with $n$ vertices such that all its leaves are adjacent to the same vertex is isomorphic to $\Star_n$.
\end{proposition}

\begin{proof}
We can assume $n \geq 5$, as the statement is trivial for $n \leq 4$. The leaves of $\mathcal T$ and $\ell'$ yield a subgraph of $\mathcal T$ isomorphic to a star graph $\mathcal S$ with central vertex $\ell'$, which we shall assume is not all of $\mathcal T$. Since $\mathcal T$ is connected, there exists $v \in V(\mathcal T)$ adjacent to $\ell'$ and not a leaf: take $\mathcal T'$ to be $\mathcal T$ minus the edge $\{v, \ell'\}$, and $C$ the component of $\mathcal T'$ containing $v$, so $C$ is itself a tree with at least two vertices, and thus with at least two leaves by Proposition \ref{trivial_leaf_prop}. The set $V(C)$ cannot have any vertices of $\mathcal S$ (if it did, there would exist a cycle in $\mathcal T$), so all leaves of $C$ are not adjacent to $\ell'$. Upon adding the edge $\{v, \ell'\}$ to $\mathcal T'$ to achieve $\mathcal T$, at least one of the leaves of $C$ remains a leaf in $\mathcal T$ not adjacent to $\ell'$, a contradiction.
\end{proof}

\noindent We now prove the main lemma that we invoke for deriving Theorem \ref{conjec_7.1}.
\begin{lemma} \label{r=1}
Let $X$ be a graph with vertex set $V(X) = [n]$ for $n \geq 5$ and edge set $E(X) = \{\{i, i+1\}: 1 \leq i \leq n-1\} \cup \{1, w\} \cup \{n, v\}$, satisfying $2 \leq v \leq \lfloor \frac{n}{2} \rfloor < w \leq n-1$ and $v+[(n+1)-w] < n$. Let $Y$ be a graph on $n$ vertices such that $\overline{Y}$ is a forest with at least two trees. Take permutations $\sigma, \sigma' \in \mathfrak S_n$ such that $\sigma(1) = \sigma'(n)$, $\sigma(n) = \sigma'(1)$, $\sigma(i) = \sigma'(i)$ for all $i \in [n] \setminus \{1, n\}$, and $\{\sigma(1), \sigma(n)\} \in E(Y)$. Then there exists a sequence of $(X,Y)$-friendly swaps from $\sigma$ to $\sigma'$.
\end{lemma}

\begin{proof}
If $\overline{Y}$ has an isolated vertex, the result follows directly from Proposition \ref{isol_vx}. Henceforth assume that none of the components of $\overline{Y}$ are isolated vertices, implying that $\overline{Y}$ contains at least four leaves from its (at least) two trees of size at least $2$. Take one of these four leaves, denoted $\sigma(\ell) \in V(\overline{Y})$, and let $\sigma(\ell') \in V(\overline{Y})$ denote the value adjacent to $\sigma(\ell) \in \overline{Y}$, or the unique vertex $\sigma(\ell)$ fails to swap with. We break into cases depending on whether or not $\sigma(\ell)$ and $\sigma(\ell')$ lie in $\sigma(V(\mathcal C))$ or not.

In what follows, sequences of transpositions are understood to be friendly swaps executed from left to right, and values in any particular transposition are the vertices of $X$ corresponding to the given swap. We shall refer to the vertices $\sigma(\ell') \in V(\overline{Y})$ as ``non-commuters", and the notion of achieving $\sigma'$ from $\sigma$ via some sequence of $(X, Y)$-friendly swaps as showing the ``exchangeability" of $\sigma(1)$ and $\sigma(n)$. Finally, there are many similar cases for which arguing the exchangeability of $\sigma(1)$ and $\sigma(n)$ is entirely analogous; proofs are provided for only one such setting.

\paragraph{Case 1: $\sigma(\ell') \notin \sigma(V(\mathcal C))$.} Here, we have $\ell' \in \mathcal B = \{v+1, \dots, w-1\}$. We further split into cases based on the location of $\ell$.

\subparagraph{Case 1.1: $\sigma(\ell) \in \sigma(V(\mathcal C))$.} The exchangeability of $\sigma(1)$ and $\sigma(n)$ follows immediately by Theorem \ref{orig_result} applied to the subgraph $\FS(\mathcal C, Y|_{\sigma(V(\mathcal C))})$ of $\FS(X, Y)$, as detailed previously.

\subparagraph{Case 1.2: $\sigma(\ell) \notin \sigma(V(\mathcal C))$.} Both $\ell, \ell'$ lie in $\mathcal B$. We split into cases based on their relative ordering.

\medskip

\textbf{Case 1.2.1: $\ell > \ell'$.} Since $\sigma(\ell')$ is the only value $\sigma(\ell)$ fails to commute with, $\sigma(\ell)$ can swap with any value $\sigma(\ell+i)$ with $1 \leq i \leq w - \ell$. Denote $\tau$ as the sequence of swaps given by $\tau = (\ell \ \ell+1)(\ell+1 \ \ell+2) \dots (w-1 \ w)$.

\textbf{Case 1.2.2: $\ell < \ell'$.} In this setting, $\sigma(\ell)$ can swap with any $\sigma(\ell-i)$ with $1 \leq i \leq v + \ell$. Denote $\tau$ as the sequence of swaps given by $\tau = (\ell \ \ell-1)(\ell-1 \ \ell-2) \dots (v+1 \ v)$.

\medskip

\noindent In both Cases 1.2.1 and 1.2.2, perform the sequence of swaps given by $\tau$, resulting in the configuration $\sigma \circ \tau$, where $\sigma(\ell) \in (\sigma \circ \tau)(V(\mathcal C))$, $\sigma(\ell') \notin (\sigma \circ \tau)(V(\mathcal C))$. Refer to Case 1.1 on $\sigma \circ \tau$ to interchange $\sigma(1)$ and $\sigma(n)$, then perform the sequence of transpositions indicated by $\tau^{-1}$ to return all other vertices and achieve $\sigma'$.

\paragraph{Case 2: $\sigma(\ell') \in \sigma(V(\mathcal C))$.}

Henceforth, we assume that for any leaf $\sigma(\ell) \in V(\overline{Y})$, $\sigma(\ell') \in \sigma(V(\mathcal C))$ (refer to Case 1 otherwise). In showing exchangeability of $\sigma(1)$ and $\sigma(n)$ via a sequence of $(X, Y)$-friendly swaps here, we argue on the number of non-commuters $\sigma(\ell')$ that are either $\sigma(1)$ or $\sigma(n)$. In particular, we consider non-commuter vertices $\sigma(\ell')$ in connected components of $\overline{Y}$ with size at most $\lfloor \frac{n}{2} \rfloor$ (i.e. no more than half the size of $\overline{Y}$); certainly at least one such component exists in $\overline{Y}$, which has at least two components.

\subparagraph{Case 2.1: There exists some connected component $C$ of $\overline{Y}$ satisfying $|V(C)| \leq \lfloor \frac{n}{2} \rfloor$ such that $C$ has a non-commuter $\sigma(\ell') \in V(C)$ that is neither $\sigma(1)$ nor $\sigma(n)$.}

Take such a non-commuter $\sigma(\ell')$, which has $\ell' \in V(\mathcal C) = \mathcal B_1 \cup \mathcal B_2$. We prove this for the setting in which $\ell' \in \mathcal B_1 = \{w, w+1, \dots, n\}$, and split into cases based on the location of the leaf $\ell$. In particular, by assumption on $\ell'$, $\ell' \neq n$, so $\ell' \in \{w, w+1, \dots, n-1\}$.

\medskip

\textbf{Case 2.1.1: $\ell \in \mathcal B$.} We split into cases based on whether $\ell' = w$ or not. For $\ell' = w$, define $\nu$ to be the closest value $\sigma(\nu)$ to $\sigma(w)$ along $X$ in a different component of $\overline{Y}$, with $2 \leq \nu \leq n-1$. Such a $\nu$ exists: if $\sigma(1)$, $\sigma(n)$ are the only vertices in a separate component from $\sigma(w)$ in $\overline{Y}$, then $\sigma(1)$, $\sigma(n)$ are the vertices of a component of $\overline{Y}$ isomorphic to $K_2$, contradicting $\{\sigma(1), \sigma(n)\} \in E(Y)$.

\medskip

\textbf{Subcase 2.1.1.1: $\ell' \neq w$.} Observe that $\ell' \notin V(\mathcal C_1)$, as $V(\mathcal C_1) \cap \mathcal B_1 = w$. Thus, appeal to Theorem \ref{orig_result} with respect to $\FS(\mathcal C_1, Y |_{\sigma(V(\mathcal C_1))})$, which has $\sigma(\ell) \in \overline{Y |_{\sigma(V(\mathcal C_1))}}$ isolated, to interchange $\sigma(1)$ and $\sigma(v)$ via some sequence of $(\mathcal C_1, Y |_{\sigma(V(\mathcal C_1))})$-friendly swaps, then swap $\sigma(1)$, $\sigma(n)$ along the edge $\{n, v\}$; call the resulting permutation $\tau$. Since $\sigma(n) \neq \sigma(\ell')$, $\sigma(\ell)$ remains an isolated vertex in $\overline{Y |_{\tau(V(\mathcal C_1))}}$. Thus, interchange $\sigma(1)$ and $\sigma(v)$ via some other sequence of $(\mathcal C_1, Y |_{\tau(V(\mathcal C_1))})$-friendly swaps to achieve $\sigma'$.

\textbf{Subcase 2.1.1.2: $\ell' = w, \ \nu > w. \ \ $} By choice of $\nu$, $\sigma(\nu)$ can swap with any $\sigma(\nu - i)$ for $1 \leq i \leq \nu - w$. Denote $\tau$ as the sequence of swaps given by $\tau = (\nu \ \nu-1)(\nu-1 \ \nu-2) \dots (w+1 \ w)$. Upon performing $\tau$, resulting in $\sigma \circ \tau$, $\sigma(\ell') = \sigma(w)$ is in position $w+1$ (i.e. $(\sigma \circ \tau)(w+1) = \sigma(\ell')$), and thus no longer in cycle $\mathcal C_1$. Refer to Subcase 2.1.1.1 to interchange $\sigma(1)$ and $\sigma(n)$, then perform the sequence of swaps given by $\tau^{-1}$ to return everything else to its original position and achieve $\sigma'$.

\textbf{Subcase 2.1.1.3: $\ell' = w, \ \nu < w. \ \ $} By choice of $\nu$, $\sigma(\nu)$ can swap with any $\sigma(\nu + i)$ for $1 \leq i \leq w - \nu$. Denote $\tau$ as the sequence of swaps $\tau = (\nu \ \nu+1)(\nu+1 \ \nu+2) \dots (w-1 \ w)$. Upon performing $\tau$, $\sigma(\ell') = \sigma(w)$ is in position $w-1 \in \mathcal B$, so refer to Case 1 to interchange $\sigma(1)$ and $\sigma(n)$. Perform the sequence of swaps given by $\tau^{-1}$ to return all other values to their original positions, achieving $\sigma'$.

\medskip

\textbf{Case 2.1.2: $\ell \in \mathcal B_1$.} In this setting, $\ell$ and $\ell'$ both lie on ``the same side" in the cycle $\mathcal C$.

\medskip

\textbf{Subcase 2.1.2.1: $\ell < \ell'$.}
$\sigma(\ell)$ can swap with $\sigma(\ell - i)$ for $1 \leq i \leq \ell-w+1$. Perform the sequence of swaps given by $\tau = (\ell \ \ell-1)(\ell-1 \ \ell-2)\dots(w \ w-1)$: the resulting configuration $\sigma \circ \tau$ has $\sigma(\ell) \in (\sigma \circ \tau)(\mathcal B)$ and $\sigma(\ell') \in (\sigma \circ \tau)(\mathcal B_1)$ (and specifically, $\ell' \neq n$). Thus, refer to Case 2.1.1 with respect to $\sigma \circ \tau$ and interchange $\sigma(1)$ and $\sigma(n)$, then shift $\sigma(\ell)$ back into its original position via the sequence of swaps $\tau^{-1}$.

\textbf{Subcase 2.1.2.2: $\ell > \ell'$.}
Displace $\sigma(\ell')$ into $\mathcal B$ by swapping elements in a different component from $\sigma(\ell')$ in $\overline{Y}$ up into $\mathcal B_1$. Denote the sequence of transpositions that achieves this by $\tau$: such a sequence $\tau$ exists, as we must have at most $n-w < n - \lfloor \frac{n}{2} \rfloor = \lceil \frac{n}{2} \rceil$ positions along $\mathcal B_1$ that map to some element in a different component in $\overline{Y}$ from $\sigma(\ell')$ (excluding $\sigma(\ell)$ from all $n-w+1$ vertices in $\mathcal B_1$), and we can position at least $\lceil \frac{n}{2} \rceil - 1$ vertices up into $\mathcal C$ along the vertices of $\mathcal B_1$ (excluding potentially $\sigma(1)$). Here, $\sigma(\ell') \notin (\sigma \circ \tau)(V(\mathcal C))$, so interchange $\sigma(1)$ and $\sigma(n)$ by appealing to Case 1 with respect to $\sigma \circ \tau$, and swap all values that we have shifted back into place by taking the sequence of swaps given by $\tau^{-1}$.

\medskip

\textbf{Case 2.1.3: $\ell \in \mathcal B_2$.} Here, $\ell$ and $\ell'$ lie ``on different sides" of the cycle $\mathcal C$.
\medskip

\textbf{Subcase 2.1.3.1: $\ell \neq 1$.}
Here, $\sigma(\ell)$ can swap with any $\sigma(\ell + i)$ for $1 \leq i \leq v-\ell+1$. Perform the sequence of swaps given by $\tau = (\ell \ \ell+1) \dots (v-1 \ v)(v \ v+1)$: the resulting configuration $\sigma \circ \tau$ has $\sigma(\ell) \in (\sigma \circ \tau)(\mathcal B)$ and $\sigma(\ell') \in (\sigma \circ \tau)(\mathcal B_1)$. Thus, refer to Case 2.1.1 with respect to $\sigma \circ \tau$ to interchange $\sigma(1)$ and $\sigma(n)$, then shift $\sigma(\ell)$ back into its original position by performing the sequence of swaps given by $\tau^{-1}$.

\textbf{Subcase 2.1.3.2: $\ell = 1$.} As in Subcase 2.1.2.2, displace $\sigma(\ell')$ into $\mathcal B$ by swapping vertices from a different component from $\sigma(\ell')$ in $\overline{Y}$ up into $\mathcal B_1$. Denote the sequence of transpositions that achieves this by $\tau$: such a sequence $\tau$ exists, as we must have at most $n-w+1 \leq n - \lfloor \frac{n}{2} \rfloor = \lceil \frac{n}{2} \rceil$ (from $n-w < n-\lfloor \frac{n}{2} \rfloor$) positions along $\mathcal B_1$ that map to some element in a different component in $\overline{Y}$ from $\sigma(\ell')$, and we can position at least $n - \lfloor \frac{n}{2} \rfloor = \lceil \frac{n}{2} \rceil$ vertices up into $\mathcal C$ along the vertices of $\mathcal B_1$ (since $\sigma(1)$ is in the same component as $\sigma(\ell')$, we can swap up all such vertices). Here, $\sigma(\ell') \notin (\sigma \circ \tau)(V(\mathcal C))$, so interchange $\sigma(1)$ and $\sigma(n)$ by appealing to Case 1 with respect to $\sigma \circ \tau$, and swap all values that we have shifted back into place by taking the sequence of swaps given by $\tau^{-1}$.

\subparagraph{Case 2.2: Any connected component $C$ of $\overline{Y}$ satisfying $|V(C)| \leq \lfloor \frac{n}{2} \rfloor$ has all non-commuters $\sigma(\ell') \in V(C)$ equal to either $\sigma(1)$ or $\sigma(n)$.} This case concerns precisely all remaining settings not studied by Case 2.1 above. We split into cases based on whether both $\sigma(1)$ and $\sigma(n)$ correspond to such a non-commuting vertex, or only one of them does. (Certainly, at least one must, as there exists at least one component $C$ of $\overline{Y}$ with $|V(C)| \leq \lfloor \frac{n}{2} \rfloor$, and any such component has at least one non-commuter vertex $\sigma(\ell')$.) 

\medskip

\textbf{Case 2.2.1: Both $\sigma(1)$ and $\sigma(n)$ correspond to $\sigma(\ell')$ for some appropriate $\ell$.} We can assume neither $\sigma(1)$ nor $\sigma(n)$ are leaves of their respective components in $\overline{Y}$: if one were a leaf, the corresponding component $C$ of $\overline{Y}$ would be isomorphic to $K_2$, which implies that we either contradict $2 = |V(K_2)| \leq \lfloor \frac{n}{2} \rfloor$, or $n=5$ and we can refer to Case 2.1. Certainly, at most one component $C$ of $\overline{Y}$ fails to satisfy $|V(C)| \leq \lfloor \frac{n}{2} \rfloor$, so $\overline{Y}$ has at most three components, and thus either two or three components.

\medskip

\textbf{Subcase 2.2.1.1: $\overline{Y}$ has three components.} Denote the components of $\overline{Y}$ by $C_1, C_2, C_3$, each of which is a tree with at least two leaves. Every component has a vertex of the form $\sigma(\ell')$, so one of the components (say $C_1$) has $|V(C_1)| > \lfloor \frac{n}{2} \rfloor$ (if not, there exists a non-commuter $\sigma(\ell')$ from a component with size at most $\lfloor \frac{n}{2} \rfloor$ not either $\sigma(1)$ or $\sigma(n)$). It follows that both $C_2$ and $C_3$ have exactly one non-commuter $\sigma(\ell')$ (i.e. all leaves adjacent to the same vertex), so both $C_2$ and $C_3$ are necessarily isomorphic to star graphs by Proposition \ref{leaf_prop}, whose centers correspond to $\sigma(1)$ and $\sigma(n)$ (say respectively), and whose leaves all lie in $\sigma(\{2, \dots, n-1\})$. Now, to exchange $\sigma(1)$ and $\sigma(n)$, perform the following sequences of swaps.
\begin{enumerate}
    \item Swap a leaf $\sigma(\ell_1)$ of $C_3$ into $\mathcal C_1$ (if not already there, onto vertex $w$ from $\mathcal B_1$) and call the resulting configuration $\sigma_1$. The vertex $\sigma(\ell_1)$ is isolated in $Y |_{\sigma_1(V(\mathcal C_1))}$, so apply Theorem 2.1 to $\FS(\mathcal C_1, Y |_{\sigma_1(V(\mathcal C_1))})$ to interchange $\sigma(1)$ with $\sigma(v)$, then swap $\sigma(1)$ and $\sigma(n)$ along $\{n, v\}$. If $\sigma(\ell_1)$ was swapped onto $w$, then swap $\sigma(\ell_1)$ back to its original position in $\mathcal B_1$.
    \item Swap a leaf $\sigma(\ell_2)$ of $C_2$ into $\mathcal C_1$ (if not already there, onto vertex $w$ from $\mathcal B_1$) and call the resulting configuration $\sigma_2$. The vertex $\sigma(\ell_2)$ is an isolated vertex in $Y |_{\sigma_2(V(\mathcal C_1))}$, so apply Theorem 2.1 to $\FS(\mathcal C_1, Y |_{\sigma_2(V(\mathcal C_1))})$ to interchange $\sigma(n)$ with $\sigma(v)$. If $\sigma(\ell_2)$ was swapped onto $w$, swap $\sigma(\ell_2)$ back to its original position.
\end{enumerate}
This sequence of swaps, which achieves $\sigma'$, can also be immediately adapted to the setting in which $\overline{Y}$ has two components, both of which are star graphs, and their centers correspond to $\sigma(1)$ and $\sigma(n)$. 

\medskip

\textbf{Subcase 2.2.1.2: $\overline{Y}$ has two components.} Let the two components be $C_1, C_2$. If $|V(C_1)| \leq \lfloor \frac{n}{2} \rfloor$ and $|V(C_2)| \leq \lfloor \frac{n}{2} \rfloor$, then $C_1$ and $C_2$ each have exactly one non-commuter corresponding to $\sigma(1)$ and $\sigma(n)$, so $C_1$ and $C_2$ are both isomorphic to star graphs, and we argue as in Subcase 2.2.1.1. Now assume $|V(C_1)| > \lfloor \frac{n}{2} \rfloor$, so $|V(C_2)| \leq \lfloor \frac{n}{2} \rfloor$. The component $C_2$ has exactly two non-commuters $\sigma(\ell')$, namely $\sigma(1)$ and $\sigma(n)$: say they correspond to leaves $\sigma(\ell_1)$ and $\sigma(\ell_n)$, respectively, where we have $2 \leq \ell_1, \ell_n \leq n-1$. To exchange $\sigma(1)$ and $\sigma(n)$, perform the following sequences of swaps.
\begin{enumerate}
    \item Swap $\sigma(\ell_n)$ into $\mathcal C_1$ (if not already there, onto vertex $w$ from $\mathcal B_1$), and call the resulting configuration $\sigma_1$. The vertex $\sigma(\ell_n)$ is isolated in $Y |_{\sigma_1(V(\mathcal C_1))}$, so apply Theorem 2.1 to $\FS(\mathcal C_1, Y |_{\sigma_1(V(\mathcal C_1))})$ to interchange $\sigma(1)$ with $\sigma(v)$, then swap $\sigma(1)$ and $\sigma(n)$ along $\{n, v\}$. If moved initially, swap $\sigma(\ell_n)$ back to its original position in $\mathcal B_1$.
    \item Swap $\sigma(\ell_1)$ into $\mathcal C_1$ (if not already there, onto vertex $w$ from $\mathcal B_1$), and call the resulting configuration $\sigma_2$. The vertex $\sigma(\ell_1)$ is isolated in $Y |_{\sigma_2(V(\mathcal C_1))}$, so apply Theorem 2.1 to $\FS(\mathcal C_1, Y |_{\sigma_2(V(\mathcal C_1))})$ to interchange $\sigma(1)$ with $\sigma(v)$. If moved initially, swap $\sigma(\ell_1)$ back to its original position in $\mathcal B_1$.
\end{enumerate}

\medskip

\textbf{Case 2.2.2: Either $\sigma(1)$ or $\sigma(n)$, but not both, corresponds to $\sigma(\ell')$ for some appropriate $\ell$.} In this case, $\overline{Y}$ must have exactly two components, which we denote $C_1$ and $C_2$. Certainly, at least one of them (say $C_2$) has $|V(C_2)| \leq \lfloor \frac{n}{2} \rfloor$, and $|V(C_1)| > \lfloor \frac{n}{2} \rfloor$. Thus, $C_2$ must be isomorphic to a star graph; either $\sigma(1)$ or $\sigma(n)$ corresponds to the central vertex of $C_2$, while the other is some vertex in $C_1$. Here, take $\sigma(1)$ as the center of the star graph, and $\sigma(n)$ as some vertex in $C_1$. (The proof for the other case is entirely analogous.) Consider all non-commuter vertices in $C_1$ that are not $\sigma(n)$: if we could displace one to $\mathcal B$ by swapping vertices in $C_2$, Case 1 gives exchangeability of $\sigma(1)$ and $\sigma(n)$, so assume this is not possible. We can also assume there exists some non-commuting vertex and corresponding leaf in $C_1$ both not $\sigma(n)$, as otherwise $C_1$ is also isomorphic to a star graph with center $\sigma(n)$, a setting addressed by Subcase 2.2.1.1. Call the non-commuting vertex and leaf $\sigma(\ell_a')$ and $\sigma(\ell_a)$, respectively, with $2 \leq \ell_a, \ell_a' \leq n-1$, and consider the values of $\ell_a$ and $\ell_a'$: by performing sequences of swaps analogous to those provided in Case 2.1, the only problematic relative ordering is that in which $\ell_a, \ell_a' \in \mathcal B_1$ or $\ell_a, \ell_a' \in \mathcal B_2$, and $\ell_a$ lies ``above" $\ell_a'$ (i.e. the two cases given by $\ell_a, \ell_a' \in \mathcal B_1$ and $\ell_a > \ell_a'$, or $\ell_a, \ell_a' \in \mathcal B_2$ and $\ell_a < \ell_a'$).

\medskip

If $\ell_a, \ell_a' \in \mathcal B_1$, perform the following sequences of swaps to interchange $\sigma(1)$ and $\sigma(n)$.
\begin{enumerate}
    \item Move all leaves of $C_2$ up into $\mathcal B_1$, so that $\sigma(\ell_a), \sigma(\ell_a')$ remain in $\mathcal B_1$ by assumption. Call the resulting configuration $\sigma_1$: any leaf vertex of $C_2$ is isolated in $Y |_{\sigma_1(V(\mathcal C_2))}$, so apply Theorem 2.1 to $\FS(\mathcal C_2, Y |_{\sigma_1(V(\mathcal C_2))})$ to interchange $\sigma(\ell_a)$ with $\sigma(\ell_a')$. From here, we can interchange $\sigma(1)$ and $\sigma(n)$. 
    \item To interchange $\sigma(\ell_a)$ and $\sigma(\ell_a')$ back, perform $\tau = (n \ v)(v \ v-1)$, resulting in $\sigma_2$, so any leaf vertex of $C_2$ is again isolated in $Y |_{\sigma_2(V(\mathcal C_2))}$. Apply Theorem 2.1 to $\FS(\mathcal C_2, Y |_{\sigma_2(V(\mathcal C_2))})$ to interchange $\sigma(\ell_a)$ with $\sigma(\ell_a')$, then perform $\tau^{-1}$. Move all leaves of $C_2$ back to their original positions to achieve $\sigma'$.
\end{enumerate}

Similarly, if $\ell_a, \ell_a' \in \mathcal B_2$, perform the following sequences of swaps to interchange $\sigma(1)$ and $\sigma(n)$.
\begin{enumerate}
    \item Move all leaves of $C_2$ up into $\mathcal B_2$, so that $\sigma(\ell_a), \sigma(\ell_a')$ remain in $\mathcal B_2$ by assumption. Perform $\tau = (1 \ w)(w \ w+1)$: call the resulting configuration $\sigma_1$. Any leaf vertex of $C_2$ is isolated in $Y |_{\sigma_1(V(\mathcal C_1))}$, so apply Theorem 2.1 to $\FS(\mathcal C_1, Y |_{\sigma_1(V(\mathcal C_1))})$ to interchange $\sigma(\ell_a)$ with $\sigma(\ell_a')$. Perform $\tau^{-1}$, then interchange $\sigma(1)$ and $\sigma(n)$. 
    \item Say the current configuration is $\sigma_2$. To interchange $\sigma(\ell_a)$ and $\sigma(\ell_a')$ back, all leaves of $C_2$ remain isolated in $Y |_{\sigma_2(V(\mathcal C_1))}$, so apply Theorem 2.1 again to $\FS(\mathcal C_1, Y |_{\sigma_1(V(\mathcal C_1))})$. Then move all leaves of $C_2$ back to their original positions to achieve $\sigma'$.
\end{enumerate}

\noindent This shows exchangeability of $\sigma(1)$ and $\sigma(n)$ to achieve $\sigma'$ in all cases, completing the proof of the lemma.
\end{proof}

\begin{remark}
Lemma \ref{r=1} assumes a looser condition on $Y$ than requiring $\overline{Y}$ to be a forest with trees of jointly coprime size: we merely require $\overline{Y}$ to be a forest containing at least two trees, and with no constraints on their sizes. This will be important in the proof of the more general $r \geq 2$ case.
\end{remark}

\begin{theorem}
Let $X$ be a biconnected graph on $n \geq 3$ vertices with an open ear decomposition with at most one ear, and $Y$ be a graph on $n$ vertices such that $\overline{Y}$ is a forest with trees $\mathcal T_1, \dots, \mathcal T_r$ with $\gcd(|V(\mathcal T_1)|, \dots, |V(\mathcal T_r)|) = 1$.
Then $\FS(X, Y)$ is connected.
\end{theorem}

\begin{proof}
As remarked previously, the result immediately follows for $n = 3$, $n = 4$, $X$ Hamiltonian, or when $\overline{Y}$ has an isolated vertex. For $n \geq 5$ and for any other such $X$ and $Y$, $\overline{Y}$ has at least two connected components from the coprimality condition, and thus at least four leaves. The biconnected graph $X$ has an open ear decomposition with one ear, and is isomorphic to a graph of the form studied in Theorem \ref{r=1}. Here, any edge $\{\sigma, \sigma'\} \in E(\FS(\Cycle_n, Y))$ that fails to be in $E(\FS(X, Y))$ must have been achieved by a $(\Cycle_n, Y)$-friendly swap across the edge $\{1, n\}$. Specifically, we must have that $\sigma(1) = \sigma'(n), \ \sigma(n) = \sigma'(1), \ \sigma(i) = \sigma'(i)$ for all $i \in [n] \setminus \{1, n\}$, and that $\{\sigma(1), \sigma(n)\} \in E(Y)$. Lemma \ref{r=1} guarantees the existence of a sequence of $(X,Y)$-friendly swaps from $\sigma$ to $\sigma'$. Thus, the vertices incident to any edge in $\FS(\Cycle_n, Y)$ but not in $\FS(X, Y)$ are connected via some other sequence of $(X, Y)$-friendly swaps in $\FS(X, Y)$. Since $\FS(\Cycle_n, Y)$ is connected, it follows that $\FS(X, Y)$ is also connected for any such $X$.
\end{proof}

\subsection{$r \geq 2$}

We introduce two lemmas for the proof of the general case. The first shows that we have some freedom in how we construct an open ear decomposition of a biconnected graph $X$ by starting with any arbitrary cycle subgraph of $X$. The second gives a well-known equivalent characterization of biconnectivity.

\begin{lemma} \label{init_cycle}
Let $X$ be a biconnected graph with $r$ ears, and let $P_0$ be any simple cycle in the graph $X$. Then there exists an open ear decomposition of $X$ with $P_0$ as the initial simple cycle. 
\end{lemma}

\begin{proof}
Refer to Algorithm 1 in \cite{schmidt2013simple}. As elaborated in the statements of Theorems 2 and 3 of this work, this algorithm will determine, when given some biconnected graph $X$, an open ear decomposition of $X$. Say that we want to construct an open-ear decomposition of $X$ that has the cycle subgraph $P_0$ of $X$ as the initial simple cycle. To achieve this, we can preferentially construct a depth-first search tree $\mathcal T$ (requested in step 1 of Algorithm 1) to be such that the root $r$ is a vertex of $P_0$, and $\mathcal T$ is constructed by moving around one direction of this cycle $P_0$. We shall assume that the remaining edge in $P_0$ will be a backedge $e$ oriented away from $r$, while all other edges of $P_0$ are oriented towards $r$ (as they are tree edges, or edges in $\mathcal T$). From here, as detailed in \cite{schmidt2013simple}, we begin constructing the open ear decomposition by taking backedges starting at the vertex $r$ (since $r$ is certainly least in the depth-first index of the tree $\mathcal T$, which is rooted at $r$). In particular, we can begin by tracing along the backedge $e$, which will yield the first cycle in our open-ear decomposition to be exactly $P_0$. 
\end{proof}

\begin{lemma}[\cite{Harary1969}]
A graph $G$ is biconnected if and only if for any two vertices $v, w \in V(G)$ there exists a cycle subgraph in $G$ containing the vertices $v$ and $w$.
\end{lemma}

Take biconnected $X$ with $r \geq 2$ ears and open ear decomposition $[P_0, P_1, \dots, P_r]$, and let the endpoints of  $P_r$ be denoted $v, w$. Call the graph consisting of strictly the first $r-1$ ears $X_{r-1}$, which is itself a biconnected graph with $r-1$ ears and open ear decomposition $[P_0, P_1, \dots, P_{r-1}]$.
Take any cycle $P_0'$ that is a subgraph of $X_{r-1}$ and contains the vertices $v$ and $w$, and construct a new open ear decomposition $[P_0', P_1', \dots, P_{r-1}']$ of $X_{r-1}$ that has $P_0'$ as its initial cycle, so $[P_0', P_r, P_1', \dots, P_{r-1}']$ is an open ear decomposition of $X$. 

Consider the biconnected subgraph $\Tilde{X} = [P_0', P_r]$ of $X$ with $m \leq n$ vertices, which we can assume is non-Hamiltonian.\footnote{The aim of the proceeding argument is to induct on the number of ears $r$. The induction would be trivial if $\tilde{X}$ were Hamiltonian, since in this setting $X$ only adds edges to a biconnected graph with $r-1$ ears. In particular, we have $|V(\tilde{X})| \geq 5$.} By Proposition \ref{ear_fold}, $\Tilde{X}$ is isomorphic to a graph constructed from $\Cycle_m$ by removing the edge $\{1, m\}$ and adding two crossing edges. We can thus understand $X$ as an $(r-1)$-ear biconnected graph $X_0$ with one edge in the initial cycle removed (denote this $\{v, w\}$; one of these corresponds to an inner vertex of $P_r$) and two more edges added.

\begin{proposition} \label{geq2_isom}
Let $X$ be any biconnected graph with $r \geq 2$ ears, and the graph $X_0$ and vertices $v,w$ be as described above. There exists an open ear decomposition of $X$ with cycle and first ear $P_0^* = [P_0, P_1]$ and outer ears $[P_2, \dots, P_r]$, where $P_0^*$ is a graph of the form described in Proposition \ref{ear_fold}. Furthermore, at least one of the vertices $v$ or $w$ has degree $2$ in $X_0$.
\end{proposition}

\noindent We now establish the appropriate analogue of Lemma \ref{r=1} for the general case.
\begin{lemma} \label{r_geq_2}
Let $X$ be a biconnected graph on $n \geq 5$ vertices with $r \geq 2$, and let $Y$ be such that $\overline{Y}$ is a forest with at least two trees. 
Let $v, w$ be the two endpoints of the edge that was removed from $X_0$ in constructing $X$.
Take $\sigma, \sigma'$ such that $\sigma(v) = \sigma'(w)$, $\sigma(w) = \sigma'(v)$, $\sigma(i) = \sigma'(i)$ for all $i \in V(X) \setminus \{v, w\}$, and $\{\sigma(v), \sigma(w)\} \in E(Y)$. Then there exists a sequence of $(X,Y)$-friendly swaps from $\sigma$ to $\sigma'$.
\end{lemma}

\begin{proof}
Let $[P_0, P_1, \dots, P_r]$ denote the open ear decomposition of $X$, and let $\Tilde{X}$ denote the biconnected subgraph with open ear decomposition $[P_0, P_1]$. Consider the subgraph $\FS(\Tilde{X}, Y |_{\sigma(\Tilde{X})})$ of $\FS(X, Y)$. If the subgraph $\overline{Y |_{\sigma(\Tilde{X})}}$ of $\overline{Y}$ consists of at least two trees, then we can invoke Lemma \ref{r=1} on $\FS(\Tilde{X}, Y |_{\sigma(\Tilde{X})})$ to interchange $\sigma(v)$ and $\sigma(w)$ to achieve $\sigma'$. If $\overline{Y |_{\sigma(\Tilde{X})}}$ has only one connected component, then there exists a sequence of $(X, Y)$-friendly swaps $\tau$ that moves an element in some other component of $\overline{Y}$ down into $\Tilde{X}$ without moving $\sigma(v)$ and $\sigma(w)$, since either $v$ or $w$ has degree $2$ in $X_0$ (it cannot be that all $P_i, i \geq 2$ have $v$ and $w$ as their two endpoints). The resulting $\sigma \circ \tau$ gives $\overline{Y _{(\sigma \circ \tau)(\Tilde{X})}}$ as a forest with at least two trees: invoke Lemma \ref{r=1} with respect to $\FS(\Tilde{X}, Y |_{\sigma(\Tilde{X})})$ to interchange $\sigma(v)$ and $\sigma(w)$, and perform the sequence of swaps $\tau^{-1}$ to return all other elements to their original positions, achieving $\sigma'$.
\end{proof}

\begin{theorem} \label{main_cor}
Let $Y$ be a graph on $n \geq 3$ vertices such that $\overline{Y}$ is a forest with trees $\mathcal T_1, \dots, \mathcal T_r$ such that $\gcd(|V(\mathcal T_1)|, \dots, |V(\mathcal T_r)|) = 1$. Assume that for all biconnected graphs $X$ on $n$ vertices with at most $r-1$ ears, $\FS(X, Y)$ is connected. Then for any biconnected graph $X$ on $n$ vertices with an open ear decomposition consisting of $r$ ears, $\FS(X, Y)$ is connected.
\end{theorem}

\begin{proof}
Any graph $X$ with $r$ ears is understood as constructed by taking a particular $(r-1)$-ear biconnected graph $X_0$, removing the edge $\{v, w\}$ and adding two more edges to the initial cycle. Any edge $\{\sigma, \sigma'\} \in E(\FS(X_0, Y))$ not in $E(\FS(X, Y))$ must have $\sigma(v) = \sigma'(w)$, $\sigma(w) = \sigma'(v)$, and $\sigma(i) = \sigma'(i)$ for all $i \in V(X) \setminus \{v, w\}$. Lemma \ref{r_geq_2} shows that any such $\sigma, \sigma'$ are connected in $\FS(X, Y)$, so that $\FS(X, Y)$ is connected.
\end{proof}

\noindent Invoking Theorem \ref{main_cor} inductively on $r$ completes the proof of Theorem $\ref{conjec_7.1}$.

\section{Girth of $\FS(X, Y)$}

We now study the notion of girth in $\FS(X, Y)$, which corresponds to a sequence of $(X, Y)$-friendly swaps such that we start and end in the same configuration. In what follows, we shall be motivated towards an exact characterization of the girth of $\FS(X, Y)$ in terms of the structure of the graphs $X$ and $Y$.

\subsection{Basic Properties + Problem Setup}

The graph $\FS(X, Y)$ is bipartite, so $g(\FS(X,Y))$ is even if finite. As remarked in \cite{defant2020friends}, if $X$ and $Y$ both have two disjoint edges, it follows that $g(\FS(X, Y)) = 4$. We make a stronger statement in this direction.

\begin{proposition} \label{girth_4}
$g(\FS(X, Y)) = 4$ if and only if $X$ and $Y$ each have at least two disjoint edges, or both contain $K_3$ as a subgraph.
\end{proposition}

\begin{proof}
The result is clear when $X$ and $Y$ both have $K_3$ as a subgraph. For the converse, assume $X$ does not have two disjoint edges nor $K_3$ as a subgraph, so $X$ is the disjoint union of a star graph and a (possibly empty) collection of isolated vertices. Consider a cycle $\mathcal C$ of size $4$ in $\FS(X, Y)$, with vertices $V(\mathcal C) = \{\sigma_1, \sigma_2, \sigma_3, \sigma_4\}$ (as ordered in $\mathcal C$); assume $\sigma_1(v)$ denotes the mapping of the center $v \in V(X)$ of the star. $\sigma_2$ and $\sigma_4$ follow from $\sigma_1$ by swapping $\sigma_1(v)$ with distinct positions, from which we cannot construct $\sigma_3$ completing $\mathcal C$. Thus, $X$ has two disjoint edges or $K_3$ as a subgraph; since $\FS(X,Y) \cong \FS(Y,X)$, so does $Y$. 
\end{proof}

\begin{remark} \label{n=3}
If $n=3$, it is easy to check that either $g(\FS(X,Y)) = 4$ (when $X = Y = K_3$), $g(\FS(X,Y)) = 6$ (when either $X$ or $Y$ is $K_3$, and the other is $K_3$ with one edge removed; here $\FS(X,Y) \cong \Cycle_6$), or $\FS(X,Y)$ is acyclic (when neither of the above two cases hold).
\end{remark}

\noindent Henceforth assume $n \geq 4$ and $g(\FS(X, Y)) \geq 6$, so that (without loss of generality) $Y$ does not have two disjoint edges, yielding $Y$ as the disjoint union of a star graph and a collection of isolated vertices. By Proposition \ref{fs_decomp} and mentioned in \cite{defant2020friends}, it suffices to consider the setting where $X$ is connected and $Y = \Star_n$.

\begin{proposition} \label{inf_girth}
Let $X$ be acyclic ($g(X) = \infty$). Then $\FS(X, \Star_n)$ is also acyclic ($g(\FS(X, \Star_n)) = \infty$).
\end{proposition}
\begin{proof}
We prove the contrapositive, or $g(\FS(X, \Star_n)) < \infty \implies g(X) < \infty$. Take some cycle subgraph $\mathcal C$ of $\FS(X, \Star_n)$, and define the set $\mathcal S$ to be all indices $i \in [n]$ such that there exist $\sigma, \tau \in V(\mathcal C)$ with $\sigma(i) \neq \tau(i)$ (in other words, $\mathcal S$ consists of all vertices in $X$ that are ``involved" in the construction of the cycle $\mathcal C$ in $\FS(Y, \Star_n)$; in particular, $\mathcal S$ corresponds precisely to all vertices $i \in V(X)$ for which $\sigma(i) = n$, or the center of the star graph, at some permutation $\sigma \in V(\mathcal C)$). 
From here, fix some arbitrary $\sigma \in \mathcal S$, and consider the induced subgraphs $X |_{\mathcal S}$ of $X$. It follows directly from construction of the set $\mathcal S$ that the induced subgraph $X |_{\mathcal S}$ of $X$ is connected and $\delta(X |_{\mathcal S}) \geq 2$. It is well-known that any such graph (connected and with minimal degree at least $2$) necessarily contains a cycle subgraph, so we conclude that $g(X) < \infty$.
\end{proof} 

\noindent Proposition \ref{inf_girth} shows that $\FS(X, \Star_n)$ is acyclic whenever $X$ is acyclic. It should thus seem natural to proceed by considering the number of cycle subgraphs of $X$. Indeed, in the case that $X$ has precisely one cycle subgraph, we can achieve some immediate statements.

\begin{lemma} \label{star_cycle}
Every connected component of $\FS(\Cycle_n, \Star_n)$ is isomorphic to $\Cycle_{n(n-1)}$.
\end{lemma}
\begin{proof}
Let $\sigma = \sigma(1)\dots \sigma(n)$ be in $\mathfrak S_n$, and consider the component $\mathcal C$ of $\FS(\Cycle_n, \Star_n)$ with $\sigma \in V(\mathcal C)$. Without loss of generality, say $\sigma(1) = n$, the central vertex of $\Star_n$ ($\mathcal C$ must have such a permutation). Let $[n(n-1)] = \{1, 2, \dots, n(n-1)\}$ denote the vertex set of $\Cycle_{n(n-1)}$, and define $\varphi: V(\Cycle_{n(n-1)}) \to V(\mathcal C)$ by defining $\varphi(i)$ to be the permutation achieved by starting from $\sigma$ and swapping $\sigma(1) = n$ rightward $i$ times (for example, $\varphi(1) = \sigma(2)\sigma(1)\dots \sigma(n)$). It follows easily that $\varphi$ is a graph isomorphism.
\end{proof}

\begin{proposition} \label{quad_bd}
If $g(X) < \infty$, we have $g(\FS(X, \Star_n)) \leq g(X) \cdot (g(X)-1)$.
\end{proposition}
\begin{proof}
Let $\mathcal C \cong \Cycle_{g(X)}$ be a cycle subgraph in $X$. We have that $\FS(\mathcal C, \Star_{g(X)})$ is a subgraph of $\FS(X, \Star_n)$, from which Lemma \ref{star_cycle} gives the desired upper bound on $g(\FS(X, \Star_n))$. 
\end{proof}

\noindent Combining Proposition \ref{inf_girth} and Corollary \ref{quad_bd} yields the following.

\begin{corollary} \label{girth_inf}
$g(\FS(X, \Star_n)) = \infty$ if and only if $g(X) = \infty$.
\end{corollary}

\subsubsection{Main Problem}

\noindent The authors of \cite{defant2020friends} left the characterization of the girth of $\FS(X, \Star_n)$ for connected graphs $X$ open-ended. Toward this, we pose the following problem, and shall make substantial progress towards resolving it.

\begin{problem} \label{girth_problem}
Find a precise description of the set of simple graphs $\mathcal G$ with finite girth such that for all $n \geq 3$, the following two statements hold.
\begin{itemize}
    \item Any $n$-vertex connected graph $X$ with $g(X) < \infty$ has that $g(\FS(X, \Star_n)) = g(\FS(\Tilde{X}, \Star_m))$ for some subgraph $\Tilde{X}$ of $X$ that is in $\mathcal G$ on $m \leq n$ vertices.
    \item If $X \in \mathcal G$, the only subgraph $\tilde{X} \in \mathcal G$ of $X$ satisfying $g(\FS(X, \Star_n)) = g(\FS(\Tilde{X}, \Star_m))$ is $\tilde{X} = X$ itself.
\end{itemize}
\end{problem}
We verify that this problem is well-posed by showing that $\mathcal G$ must be the set of all simple graphs with finite girth such that any $X \in \mathcal G$ (say on $n$ vertices) contains no proper subgraph $\tilde{X}$ on $m \leq n$ vertices with $g(\FS(X, \Star_n)) = g(\FS(\Tilde{X}, \Star_m))$ (in particular, notice that all graphs in $\mathcal G$ must be connected).\footnote{Although these are equivalent descriptions of $\mathcal G$, we pose the problem as above to emphasize that what we are fundamentally concerned with are the necessary trajectories the central vertex of $\Star_n$ executes around a graph $X$ possessing finite girth to achieve $g(\FS(X, \Star_n))$. We shall, however, appeal to this second characterization of $\mathcal G$ later in the article.} For any $n$-vertex graph $X$ with $g(X) < \infty$, either $X \in \mathcal G$ or there exists a proper subgraph $X'$ of $X$ on $m \leq n$ vertices such that $g(\FS(X, \Star_n)) = g(\FS(X', \Star_m))$. Continue similarly on $X'$ until this process necessarily bottoms out to a subgraph $\tilde{X}$ of $X$ that is contained in $\mathcal G$, so that $\mathcal G$ satisfies the requirements of Problem \ref{girth_problem}. We also see by this discussion that this constitutes \textit{the} set of graphs $\mathcal G$; certainly, any set $\mathcal G'$ of graphs satisfying the criteria of Problem \ref{girth_problem} must contain $\mathcal G$ (consider taking $X \in \mathcal G$ in Problem \ref{girth_problem}), and including any other graph in $\mathcal G'$ with finite girth that fails to be in $\mathcal G$ breaks the latter part of the statement. 

We now motivate why this problem is a natural one by observing that inherently tied to an improved understanding of $\mathcal G$ is the trajectory the central vertex in $\Star_m$ takes on a given subgraph $\tilde{X} \in \mathcal G$ to achieve $g(\FS(\tilde{X}, \Star_m))$. Indeed, our guiding aim in posing this problem is to show that not many such trajectories are necessary to characterize the girth of any such $\FS(X, \Star_n)$. The following definition corresponds to the subgraph of $X$ the central vertex of $\Star_n$ ``walks along" during a path in $\FS(X, \Star_n)$.

\begin{definition} \label{path_induced}
Let $\mathcal P = \{\sigma_i\}_{i=0}^\lambda \subset V(\FS(X, \Star_n))$ be a path in $\FS(X, \Star_n)$ ($\{\sigma_{i-1}, \sigma_i\} \in E(\FS(X, \Star_n))$ for all $i \in [\lambda]$). The \textbf{$\mathcal P$-induced subgraph of $X$}, denoted $X_{\mathcal P}$, is such that the following hold. 
\begin{itemize}
    \item $v \in V(X_{\mathcal P})$ if and only if there exists $0 \leq i \leq \lambda$ such that $\sigma_i(v) = n$.
    \item $\{v, w\} \in E(X_{\mathcal P})$ if and only if there exists $i \in [\lambda]$ such that $\sigma_{i-1}(v) = \sigma_i(w)$ and $\sigma_{i-1}(w) = \sigma_i(v)$.
\end{itemize}
\end{definition}
\noindent Now let $\mathcal C = \{\sigma_i\}_{i=0}^\lambda \subset V(\FS(X, \Star_n))$ be a cycle in $\FS(X, \Star_n)$ achieving the girth, so $\lambda = |V(\mathcal C)| = g(\FS(X, \Star_n))$, $\{\sigma_{i-1}, \sigma_i\} \in E(\FS(X, \Star_n))$ for $i \in [\lambda]$, and $\sigma_0 = \sigma_\lambda$: it follows that $g(\FS(X, \Star_n)) = g(\FS(X_{\mathcal C}, \Star_{|V(X_{\mathcal C})|}))$, and that $X_C$ is connected, $g(X_{\mathcal C}) < \infty$, and $\delta(X_{\mathcal C}) \geq 2$: these properties thus extend to any graphs in $\mathcal G$. We can thus intuitively understand the demands of Problem \ref{girth_problem} as determining the smallest collection of trajectories, represented by the corresponding subgraphs in $\mathcal G$ where all vertices and edges are traversed, that are executed by the central vertex of $\Star_n$ along cycles in graphs $\FS(X, \Star_n)$.

It follows immediately from how $\mathcal G$ was defined that for any $X$ with $g(X) < \infty$, $g(\FS(X, \Star_n))$ is precisely the minimum of $g(\FS(\tilde{X}, \Star_m))$ over all subgraphs $\tilde{X}$ of $X$ (on $m \leq n$ vertices) with $\tilde{X} \in \mathcal G$. As such, we seek a concrete description of the set $\mathcal G$, determining the types of subgraphs $n \in V(\Star_n)$ must traverse to achieve the girth.

\subsection{Subset of $\mathcal G$}

To begin our investigation of $\mathcal G$, we shall initially concern ourselves with three families of graphs, which should not seem arbitrary: barbell graphs take two cycle graphs and connect them by some (possibly trivial) path, while theta graphs take two cycle graphs and ``conjoin them" along a path. In particular, all these examples are connected, with properties $\delta(X) \geq 2$ and $0 \leq |E(X)| - |V(X)| \leq 1$.

\begin{enumerate}
    \item $X$ is a \textbf{cycle graph}, or $X = \Cycle_n$.
    \item $X$ is a \textbf{barbell graph}, which we define here as a graph $X$ that can be decomposed into three subgraphs $[\mathcal C_1, \mathcal C_2, \mathcal P]$ with $E(X) = E(\mathcal C_1) \sqcup E(\mathcal C_2) \sqcup E(\mathcal P)$ and such that one of two possibilities hold.
    \begin{itemize}
        \item $\mathcal C_1$ and $\mathcal C_2$ are vertex-disjoint cycles, and $\mathcal P$ is a path with inner vertices disjoint from $V(\mathcal C_1)$ and $V(\mathcal C_2)$, while the endpoints $v, w$ of $\mathcal P$ have $v \in V(\mathcal C_1)$, $w \in V(\mathcal C_2)$. 
        \item $\mathcal C_1$ and $\mathcal C_2$ are cycles that intersect at exactly one vertex, and $\mathcal P$ is a trivial path consisting of precisely this vertex.
    \end{itemize}
    \item $X$ is a biconnected graph with $r=1$ (henceforth called a \textbf{$\theta$-graph}, following \cite{wilson1974graph}).
\end{enumerate}

\subsubsection{Graphs with One Cycle Subgraph}

\begin{proposition} \label{one_cycle_subgph}
Let $X$ be a connected graph on $n$ vertices with precisely one cycle subgraph $\mathcal C$, and say $\mathcal C \cong \Cycle_k$ for some $k \leq n$. Then $g(\FS(X, \Star_n)) = k(k-1)$.
\end{proposition}

\begin{proof}
There exists a vertex-disjoint decomposition of $X$ into subgraphs $[\mathcal T_0 = \mathcal C, \mathcal T_1, \dots, \mathcal T_m]$ that partitions its edge set, where the $\mathcal T_i$ are trees for all $i \geq 1$. Rigorously, let $\mathcal V = \{v_1, \dots, v_m\}$ denote the set of all vertices adjacent to a vertex in $\mathcal C$, but not in $\mathcal C$. For vertex $v_i$, remove the corresponding vertex in $\mathcal C$, and refer to the component in the resulting graph containing $v_i$ by $\mathcal T_i$: this is connected and acyclic (as $\mathcal C$ was the only cycle in $X$), and thus a tree. From $\mathcal C$ being the only cycle subgraph in $X$, it follows that every vertex and edge of $X$ must be contained in some unique subgraph $\mathcal T_i$. 

Let $\mathcal C'$ be a cycle subgraph in $\FS(X, \Star_n)$ that achieves the girth. Recall that $X_{\mathcal C'}$ must have finite girth and $\delta(X_{\mathcal C'}) \geq 2$, so $X_{\mathcal C'}$ must contain $\mathcal C$ as a subgraph, and for any tree $\mathcal T_i$ with $i \geq 1$, the leaves of $\mathcal T_i$ cannot be in $X_{\mathcal C'}$. Upon removing such leaves from $\mathcal T_i$, the leaves of the resulting graph also fail to be in $X_{\mathcal C'}$. In this manner, we can continually ``prune" the leaves of each $\mathcal T_i$ to conclude that $X_{\mathcal C'} = \mathcal C$. Thus, $\mathcal C'$ also lies in the subgraph $\FS(X_{\mathcal C'}, \Star_k) \cong \FS(\mathcal C, \Star_k)$, so $g(\FS(X, \Star_n)) = k(k-1)$.
\end{proof}

\noindent In particular, this yields that any graph in $\mathcal G$ that is not a cycle has at least two cycle subgraphs.

\subsubsection{Graphs with Two Cycle Subgraphs}

We proceed to the case in which $X$ has at least two distinct cycle subgraphs. In this direction, we begin by motivating why barbell and theta graphs are natural candidates of such graphs $X$ to study. 

\begin{lemma} \label{two_cycle_subgphs}
Any connected graph $X$ with at least two cycle subgraphs must necessarily have a subgraph isomorphic to either a barbell graph or a $\theta$-graph.
\end{lemma}

\begin{proof}
Let $\mathcal C_1$ and $\mathcal C_2$ be two distinct cycle subgraphs of $X$. The existence of a barbell subgraph of $X$ is immediate if $\mathcal C_1$ and $\mathcal C_2$ share precisely one vertex. If $\mathcal C_1$ and $\mathcal C_2$ are vertex-disjoint, take $v \in V(\mathcal C_1)$ and $w \in V(\mathcal C_2)$ and any path $P_0$ from $v$ to $w$. Take the last vertex along $P_0$ and the earliest vertex in $\mathcal C_2$ after this point, and truncate $P_0$ at these respective ends. This yields a path between $\mathcal C_1$ and $\mathcal C_2$ with inner vertices disjoint from both cycles, and thus a barbell subgraph of $X$.

Now say $X$ does not contain a barbell subgraph, so any two cycle subgraphs $\mathcal C_1$ and $\mathcal C_2$ have at least two vertices in common. To construct a path $\mathcal P$ with endpoints in $V(\mathcal C_1)$ but otherwise vertex-disjoint from $V(\mathcal C_1)$, let $v \in V(\mathcal C_1) \cap V(\mathcal C_2)$. We either have that one of the two edges of $\mathcal C_2$ incident to $v$ is distinct from those in $\mathcal C_1$, or the neighbors of $v$ in $\mathcal C_2$ are also the neighbors of $v$ in $\mathcal C_1$. In the latter case, take one of the two neighbors of $v$ in $\mathcal C_1$ and $\mathcal C_2$, denote this by $v$, and proceed similarly until we eventually get $v' \in V(\mathcal C_2)$ not following the cycle in $\mathcal C_1$ (the distinctness of $\mathcal C_1$ and $\mathcal C_2$ guarantees this exists). Let $v$ be the first vertex in path $\mathcal P$, and $v'$ be the following vertex. Then proceed along $\mathcal C_2$ from $v'$ until we encounter $w \in V(\mathcal C_1) \cap V(\mathcal C_2)$ (which must eventually happen, since we assumed $\mathcal C_1$ and $\mathcal C_2$ have at least two common vertices). Then $[\mathcal C_1, \mathcal P]$ is a theta subgraph of $X$.
\end{proof}

\begin{remark} \label{girth_basics}
Section \ref{graph_families} defines $n \in V(\Star_n)$ as the central vertex of $\Star_n$. To avoid confusion, we henceforth elect to instead notate the central vertex of $\Star_n$ by $\mathfrak n$. In forthcoming discussions, we shall frequently elect to informally refer to a swap sequence\footnote{Such a swap sequence will usually be denoted $\mathcal S$, with the transpositions in $\mathcal S$ denoting pairs of vertices along which swaps occur in the $X$ graph, which always involve $\mathfrak n$. We occasionally also use $\mathcal V = \{\sigma_i\}_{i=0}^\lambda$ to refer specifically to the vertices in $\FS(X, \Star_n)$ corresponding to the configurations resulting from the swap sequence $\mathcal S$. Here, $\mathcal V$ need not be unique, as we can place all vertices $V(\Star_n) \setminus \{n\}$ in any way onto $V(X) \setminus \sigma_0^{-1}(n)$ and refer to the trajectory that $\mathfrak n$ takes along $X$. Thus, whenever it is referenced in an argument, assume $\mathcal V$ refers to any such path in $\FS(X, \Star_n)$.} that $\mathfrak n$ traces around the graph $X$ to achieve a cycle in $\FS(X, \Star_n)$. The following hold for any swap sequence $\mathcal S$ that achieves a cycle in $\FS(X, \Star_n)$.
\begin{enumerate}
    \item Denote the path taken by $y \in V(\Star_n)$ during $\sigma$ by $\mathcal P_y$. Either $\mathcal P_y$ ``folds in on itself" by retracing all edges in the opposite direction, or we can extract a simple cycle in $X$ from $\mathcal P_y$.
    \item We can assume, without loss of generality, that $\mathfrak n$ begins its traversal of $X$ on any $v \in \sigma^{-1}(V(Y))$. This corresponds to fixing $\sigma_1(v) = \mathfrak n$, or circularly shifting swaps of $\sigma$ depending on the most convenient vertex of $\sigma^{-1}(V(Y))$ at which to begin studying the traversal of $X$.
\end{enumerate}
\end{remark}

\noindent We now restrict our study to $X$ being either a barbell or a theta graph. For such cases, we can derive upper bounds on $g(\FS(X, \Star_n))$ that are linear in $g(X)$, a notable improvement from Corollary \ref{quad_bd}.

\begin{proposition} \label{barbell_bd}
Take $X$ a barbell graph with decomposition $[\mathcal C_1, \mathcal C_2, \mathcal P]$. Then $g(\FS(X, \Star_n)) \leq 2(|V(\mathcal C_1)| + |V(\mathcal C_2)| + 2|E(\mathcal P)|)$. Also, if $2(|V(\mathcal C_1)| + |V(\mathcal C_2)| + 2|E(\mathcal P)|) < \min\{|V(\mathcal C_1)|(|V(\mathcal C_1)|-1), |V(\mathcal C_2)|(|V(\mathcal C_2)|-1)\}$, then $g(\FS(X, \Star_n)) = 2(|V(\mathcal C_1)| + |V(\mathcal C_2)| + 2|E(\mathcal P)|)$.
\end{proposition}

\begin{figure}[ht]
    \centering
    \includegraphics[width=0.7\textwidth]{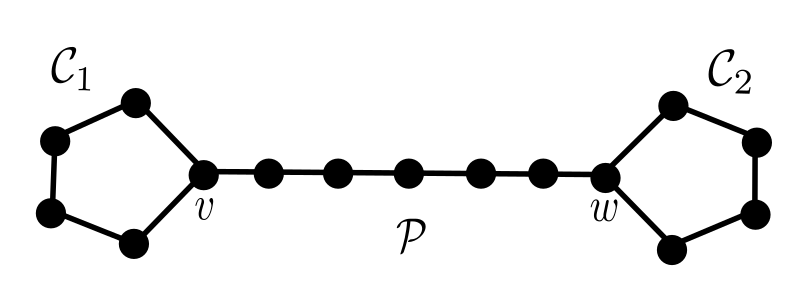}
    \caption{An illustration of a barbell graph, applying the proceeding notation. The sequence of $(X, \Star_n)$-friendly swaps starting at $v$ described by $\sigma$ yields a cycle in $\FS(X, \Star_n)$.}
    \label{fig:barbell}
\end{figure}

\begin{proof}
Denote the vertices of $\mathcal P$ by $\{x_1 = v, \dots, x_m = w\}$, in order following the path. Letting $m_1 = |V(\mathcal C_1)|$ and $m_2 = |V(\mathcal C_2)|$, denote the vertices of $\mathcal C_1$ and $\mathcal C_2$ by $\{v_1 = v, \dots, v_{m_1}\}$ and $\{w_1 = w, \dots, w_{m_2}\}$, respectively. Consider the following sequences of swaps.
\begin{align*}
    & \mathcal S_1 = (v \ v_2)(v_2 \ v_3)\dots(v_{m_1-1} \ v_{m_1})(v_{m_1} \ v) \\
    & \mathcal S_2 = (v \ x_2)(x_2 \ x_3) \dots (x_{m-1} \ w) \\
    & \mathcal S_3 = (w \ w_2)(w_2 \ w_3) \dots (w_{m_2-1} \ w_{m_2})(w_{m_2} \ w)
\end{align*}
Starting $\mathfrak n$ at $v$, perform the sequence $\sigma$ of $(X, \Star_n)$-friendly swaps given by
\begin{align*}
    \mathcal S: \mathcal S_1 \to \mathcal S_2 \to \mathcal S_3 \to \mathcal S_2^{-1} \to \mathcal S_1^{-1} \to \mathcal S_2 \to \mathcal S_3^{-1} \to \mathcal S_2^{-1}
\end{align*}
It is straightforward to confirm that $\mathcal S$ corresponds to a cycle in $\FS(X, \Star_n)$ with length $2(|V(\mathcal C_1)| + |V(\mathcal C_2)| + 2|E(\mathcal P)|)$. Now consider the setting $2(|V(\mathcal C_1)| + |V(\mathcal C_2)| + 2|E(\mathcal P)|) < \min\{|V(\mathcal C_1)|(|V(\mathcal C_1)|-1), |V(\mathcal C_2)|(|V(\mathcal C_2)|-1)\}$, and assume that there exists a swap sequence $\mathcal S'$ with length strictly less than $2(|V(\mathcal C_1)| + |V(\mathcal C_2)| + 2|E(\mathcal P)|)$ corresponding to a cycle in $\FS(X, \Star_n)$. Certainly the path $\mathcal P$ is crossed an even number of times (to return to a vertex in $V(\mathcal C_1)$, where we assume without loss of generality $\mathfrak n$ begins its traversal), so $\mathcal P$ is crossed exactly twice. But then $\mathcal S'$ must have completed some cycle in $\FS(X, \Star_n)$ strictly on the subgraph $V(\mathcal C_2)$ as we do not return there, a contradiction on $\mathcal S'$ as a sequence of swaps achieving the girth of $\FS(X, \Star_n)$.
\end{proof}

\begin{proposition} \label{theta_bd}
Let $X$ be a $\theta$-graph with cycle subgraphs $\mathcal C_1, \mathcal C_2, \mathcal C$, as depicted in the figure below.
\begin{enumerate}
    \item If all three paths contain an inner vertex, $g(\FS(\Star_n, Y)) \leq 2(|V(\mathcal C_1)| + |V(\mathcal C_2)| + |V(\mathcal C)|)$. Furthermore, if $2(|V(\mathcal C_1)| + |V(\mathcal C_2)| + |V(\mathcal C)|) < \min\{|V(\mathcal C_1)|(|V(\mathcal C_1)|-1), |V(\mathcal C_2)|(|V(\mathcal C_2)|-1), |V(\mathcal C)|(|V(\mathcal C)|-1)\}$, then $g(\FS(X, \Star_n)) = 2(|V(\mathcal C_1)| + |V(\mathcal C_2)| + |V(\mathcal C)|)$.
    \item If (exactly) one of the three paths is an edge, then $g(\FS(\Star_n, Y)) \leq 3(|V(\mathcal C_1)| + |V(\mathcal C_2)| + |V(\mathcal C)|)$. Furthermore, if $3(|V(\mathcal C_1)| + |V(\mathcal C_2)| + |V(\mathcal C)|) < \min\{|V(\mathcal C_1)|(|V(\mathcal C_1)|-1), |V(\mathcal C_2)|(|V(\mathcal C_2)|-1), |V(\mathcal C)|(|V(\mathcal C)|-1)\}$, then $g(\FS(X, \Star_n)) = 3(|V(\mathcal C_1)| + |V(\mathcal C_2)| + |V(\mathcal C)|)$.
\end{enumerate}
\end{proposition}

\begin{figure}[ht]
    \centering
    \includegraphics[height=0.3\textwidth]{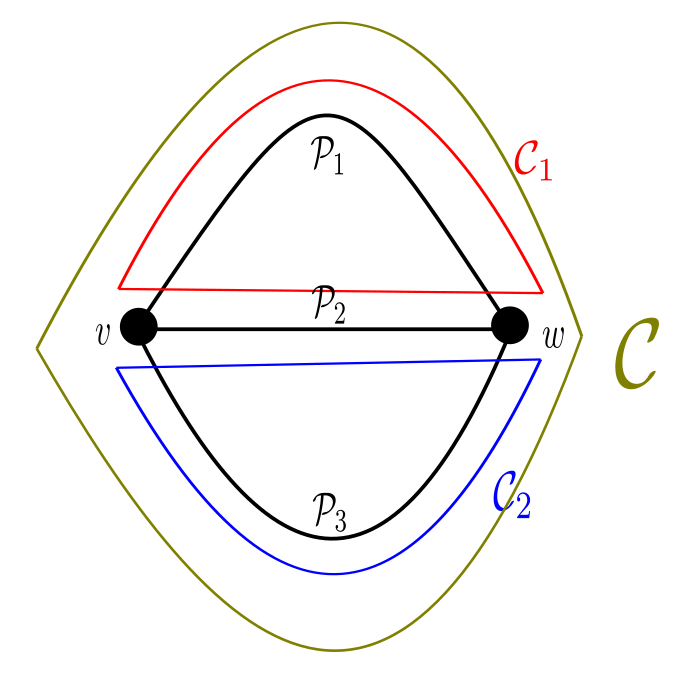}
    \caption{An illustration of a $\theta$-graph, applying the proceeding notation and labeling the cycles described above. The sequence $\mathcal S$ of $(X, \Star_n)$-friendly swaps starting at $v$ corresponds to a cycle in $\FS(X, \Star_n)$ with $\mathcal S$ differing depending on whether or not one of the three paths is an edge. In forthcoming arguments, we will often assume some ordering on the lengths of $\mathcal P_1, \mathcal P_2, \mathcal P_3$. We shall always refer to $\mathcal C_1$ as the cycle resulting from $\mathcal P_1$ and $\mathcal P_2$; similarly, $\mathcal C_2$ is the cycle resulting from $\mathcal P_2$ and $\mathcal P_3$, and $\mathcal C$ the cycle resulting from $\mathcal P_1$ and $\mathcal P_3$.}
    \label{fig:theta}
\end{figure}

\begin{proof}
Let the vertices of degree $3$ be denoted $v, w$, and label the three paths from $v$ to $w$ in $X$ by $\mathcal P_1 = \{v_1 = v, \dots, v_{m_1} = w\}$, $\mathcal P_2 = \{w_1 = v, \dots, w_{m_2} = w\}$, $\mathcal P_3 = \{x_1 = v, \dots, x_{m_3} = w\}$ in order following the path. Consider the following sequences of swaps.
\begin{align*}
    & \mathcal S_1 = (v \ v_2)(v_2 \ v_3)\dots(v_{m_1-1} \ w) \\
    & \mathcal S_2 = (v \ w_2)(w_2 \ w_3) \dots (w_{m_2-1} \ w) \\
    & \mathcal S_3 = (v \ x_2)(x_2 \ x_3) \dots (x_{m_3-1} \ w)
\end{align*}
Starting $\mathfrak n$ at $v$, perform the sequence $\sigma$ of $(X, \Star_n)$-friendly swaps given by
\begin{align*}
    \mathcal S: \mathcal S_1 \to \mathcal S_2^{-1} \to \mathcal S_3 \to \mathcal S_1^{-1} \to \mathcal S_2 \to \mathcal S_3^{-1}
\end{align*}
It is straightforward to confirm that performing $\mathcal S$ twice yields a cycle in $\FS(X, \Star_n)$ under setting (1), while performing $\mathcal S$ three times yields a cycle in $\FS(X, \Star_n)$ under setting (2); denote these $\mathcal S^{(2)}$ and $\mathcal S^{(3)}$, respectively. We show that $\mathcal S^{(2)}$ and $\mathcal S^{(3)}$ achieve the girth of $\FS(X, \Star_n)$ under their respective settings. Assume without loss of generality that $\mathfrak n$ begins its traversal on $v$: in both cases, we argue on an arbitrary traversal $\mathcal S'$ that achieves a cycle in $\FS(X, \Star_n)$. Observe that $\mathcal S'$ proceeds by selecting one of the three paths $\mathcal P_1, \mathcal P_2, \mathcal P_3$ to reach $w$, then choosing a path to return to $v$, and continuing similarly until the cycle is achieved. As such, we can represent $\mathcal S'$ as a finite word\footnote{It is perhaps more appropriate to think of this as a circular word, as we can assume the word starts on any letter and that the first traversal is from $v$ to $w$ without loss of generality: establishing the first letter corresponds to choosing where to begin following the traversal of $\mathfrak n$, while setting the first traversal from $v$ to $w$ determines the ``direction" we go around the cycle in $\FS(X, \Star_n)$ given by $\mathcal S_{\mathcal W}'$. We shall take advantage of this later.} of even length on a ternary alphabet $\{1, 2, 3\}$ with no two consecutive letters equal, where odd indices correspond to selected paths from $v$ to $w$ that $\mathfrak n$ traverses, and even ones to selected paths from $w$ to $v$. For convenience, we elect to use this representation of $\mathcal S'$, and shall denote this $\mathcal S_{\mathcal W}'$; under this representation, $\mathcal S_{\mathcal W} = 123123$. 

We assume, without loss of generality, that the lengths of the paths (i.e. number of edges, one greater than the number of inner vertices) $\mathcal P_1, \mathcal P_2, \mathcal P_3$ satisfy $p_1 \leq p_2 \leq p_3$, and denote the number of occurrences of $\{1, 2, 3\}$ in $\mathcal S_{\mathcal W}'$ by $n_1, n_2, n_3$, respectively. Also denote $\phi = (p_1, p_2, p_3)$ and $\eta = (n_1, n_2, n_3)$.

\paragraph{Case 1: $p_1, p_2, p_3 > 1$, and $2(|V(\mathcal C_1)| + |V(\mathcal C_2)| + |V(\mathcal C)|) < \min\{|V(\mathcal C_1)|(|V(\mathcal C_1)|-1), |V(\mathcal C_2)|(|V(\mathcal C_2)|-1), |V(\mathcal C)|(|V(\mathcal C)|-1)\}$.}

We show $\phi \cdot \eta \geq 2(|V(\mathcal C_1)| + |V(\mathcal C_2)| + |V(\mathcal C)|)$, for which we can assume $n_1, n_2, n_3 \neq 0$ and $\min\{|V(\mathcal C_1)|, |V(\mathcal C_2)|, |V(\mathcal C)|\} \geq 8$, so $p_2 \geq 4$. Certainly $n_i \geq 2$, as all elements originally upon inner vertices of $\mathcal P_1, \mathcal P_2, \mathcal P_3$ return to their original position, and $n_i \neq 2$, since $n_i=2$ requires traversing a cycle subgraph in the trajectory given by Lemma \ref{star_cycle}\footnote{Specifically, in between uses of one of the three particular paths, we are reduced to traversing the cycle subgraph of $X$ constructed from the other two paths in the $\theta$-graph.}, contradicting $\phi \cdot \eta < 2(|V(\mathcal C_1)| + |V(\mathcal C_2)| + |V(\mathcal C)|)$. If $p_i \geq 4$, $\mathcal P_i$ cannot be traversed three times without displacing an element on an inner vertex, so $n_2, n_3 \geq 4$. It follows that $\phi \cdot \eta \geq 2(|V(\mathcal C_1)| + |V(\mathcal C_2)| + |V(\mathcal C)|)$ unless $n_1 = 3$, in which case either $n_2 \geq 5$ or $n_3 \geq 5$ since $\mathcal S_{\mathcal W}'$ has even length, implying that we again have the inequality.

\paragraph{Case 2: $p_1 = 1$, $3(|V(\mathcal C_1)| + |V(\mathcal C_2)| + |V(\mathcal C)|) < \min\{|V(\mathcal C_1)|(|V(\mathcal C_1)|-1), |V(\mathcal C_2)|(|V(\mathcal C_2)|-1), |V(\mathcal C)|(|V(\mathcal C)|-1)\}$.}

We show $\phi \cdot \eta \geq 3(|V(\mathcal C_1)| + |V(\mathcal C_2)| + |V(\mathcal C)|)$ for any $\mathcal S'$ achieving a cycle in $\FS(X, \Star_n)$. Assume for the sake of contradiction that there exists a trajectory $\mathcal S'$ with $\lambda = \phi \cdot \eta < 3(|V(\mathcal C_1)| + |V(\mathcal C_2)| + |V(\mathcal C)|)$: $n_1, n_2, n_3 \neq 0$ and $\min\{|V(\mathcal C_1)|, |V(\mathcal C_2)|, |V(\mathcal C)|\} \geq 11$, so $p_2, p_3 \geq 10$. Arguing as in Case (1), $n_2, n_3 \geq 4$, and $n_1, n_2, n_3 \neq 3, 5$. In particular, $n_2 < 6 \implies n_2 = 4$; similarly, $n_3 < 6 \implies n_3 = 4$. Here, denote $\sigma_0$ as the first configuration in the cycle in $\FS(X, \Star_n)$ traversed via $\mathcal S$ (see the footnote in Remark \ref{girth_basics}; we refer to $\mathcal V = \{\sigma_i\}_{i=0}^\lambda \subset V(\FS(X, \Star_n))$). We must have that one of $n_1, n_2, n_3$ is less than $6$.

First assume $n_2 = 4$. Without loss of generality, the word $\mathcal S_{\mathcal W}'$ starts with $2$, corresponding to $\mathcal P_2$ traversed from $v$ to $w$; $\sigma_w'$ has four $2$s, with substrings $1313\dots$ or $3131\dots$ between occurrences of $2$. There are two instances of $2$ in odd and even indices, which must alternate, since two consecutive traversals of $\mathcal P_2$ from $v$ to $w$ yields $\mathcal C$ traversed in the trajectory given by Lemma \ref{star_cycle} so that $\sigma_0(w_3) = \sigma_\lambda(w_3)$ (here, $\sigma_0(w_3)$ would be ``pushed out" of $\mathcal P_2$ following this second traversal, and must be pushed back in by a traversal of $\mathcal P_2$ from $w$ to $v$). Thus, all other elements upon $\mathcal C$ circularly rotate around $\mathfrak n$ between traversals of $\mathcal P_2$, all by some fixed offset less than $|V(\mathcal C)|$, as $\mathcal S'$ cannot execute the trajectory of Lemma \ref{star_cycle} upon $\mathcal C$ due to $\lambda < 3(|V(\mathcal C_1)| + |V(\mathcal C_2)| + |V(\mathcal C_3)|)$ (see Figure \ref{fig:rotation}). In particular, an element ``pushed out" from $V(\mathcal P_2) \setminus \{v, w\}$ onto $\{v,w\}$ (and thus onto $V(\mathcal C)$) cannot be swapped back up in the proceeding traversal of $\mathcal P_2$. Thus, $\sigma_0(w_2)$ lies upon $V(\mathcal C)$ until the fourth traversal of $\mathcal P_2$, for which it is swapped back onto $w_2$. 

\begin{figure}[ht]
    \centering
    \includegraphics[height=0.25\textheight, width=0.6\textwidth]{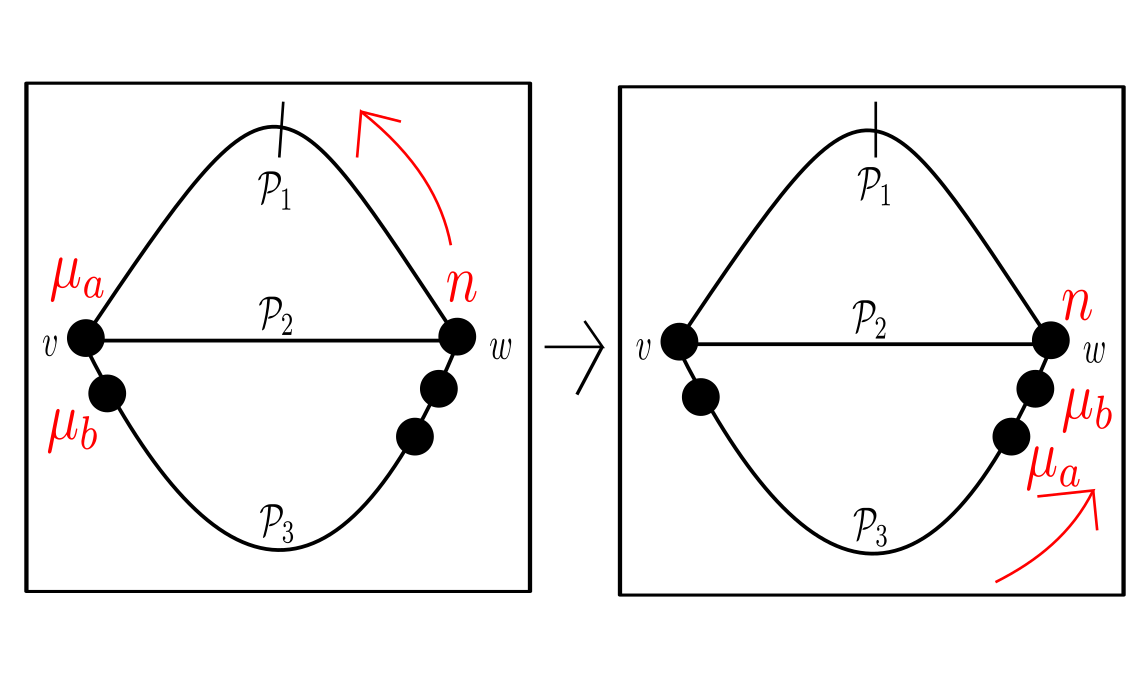}
    \caption{Between traversals of $\mathcal P_2$, $n$ moves around $\mathcal C$, and all other vertices of $\Star_n$ upon $\mathcal C$ move circularly around it; the hatch mark for $\mathcal P_1$ signifies that it is an edge. Here, we adopt the convention that vertices or subgraphs colored black are in $X$, and those colored red are in $Y$. In the diagram, $\{\mu_a, \mu_b\} \subset V(\Star_n) \setminus \{n\}$, with $n$ rotating twice around $\mathcal C$.}
    \label{fig:rotation}
\end{figure}

Take $\sigma_0(w)$, swapped onto $w_{m_2-1}$ during the first traversal of $\mathcal P_2$ and back onto $w$ during the second traversal of $\mathcal P_2$. It is clear that $\sigma_0(w)$ must remain upon $\mathcal C$ for the rest of the swap sequence (i.e. is not ``pushed in" $V(\mathcal P_2) \setminus \{v, w\}$ during the third and fourth traversals). Consider the values adjacent to $\sigma_0(w)$ upon $V(\mathcal C)$, ignoring $\mathfrak n$, after the second traversal of $\mathcal P_2$; these must lie upon vertices $x_2$ and $x_{m_3-1}$. Say that in $\sigma_0$, we have $\mu_1 = \sigma_0(x_2)$, $\mu_2 = \sigma_0(x_{m_3-1})$ (see Figure \ref{fig:theta_opt_1}). If $\sigma^{(2)} \in \mathcal V$ denotes the configuration after the second traversal, then $\sigma_0(x_2) \neq \sigma^{(2)}(x_2)$, $\sigma_0(x_{m_3-1}) \neq \sigma^{(2)}(x_{m_3-1})$. It is clear that these vertices adjacent to $\sigma_0(w)$ upon $V(\mathcal C)$ remain invariant until $n$ traverses $\mathcal P_2$ again, and at most one can change after any given traversal of $\mathcal P_2$. In particular, since $\sigma_0(w_2)$ is swapped back onto $w_2$ during the fourth traversal of $\mathcal P_2$, $\sigma_0(w)$ must be adjacent to $\sigma_0(w_2)$ after the second traversal of $\mathcal P_2$. 

From here, it is straightforward to derive a contradiction on $\phi \cdot \eta < 3(|V(\mathcal C_1)| + |V(\mathcal C_2)| + |V(\mathcal C)|)$ or $\mathcal S'$ achieving a cycle in $\FS(X, \Star_n)$ by splitting into cases based on where $\sigma_0(w_2)$ lies after the second traversal of $\mathcal P_2$. In particular, either $\mu_1$ or $\mu_2$ (for $\sigma_0(w_2)$ equal to $x_{m_3-1}$, $x_2$, respectively) lies upon $w_2$ after this second traversal. 

\begin{figure}[ht]
    \centering
    \includegraphics[width=0.9\textwidth]{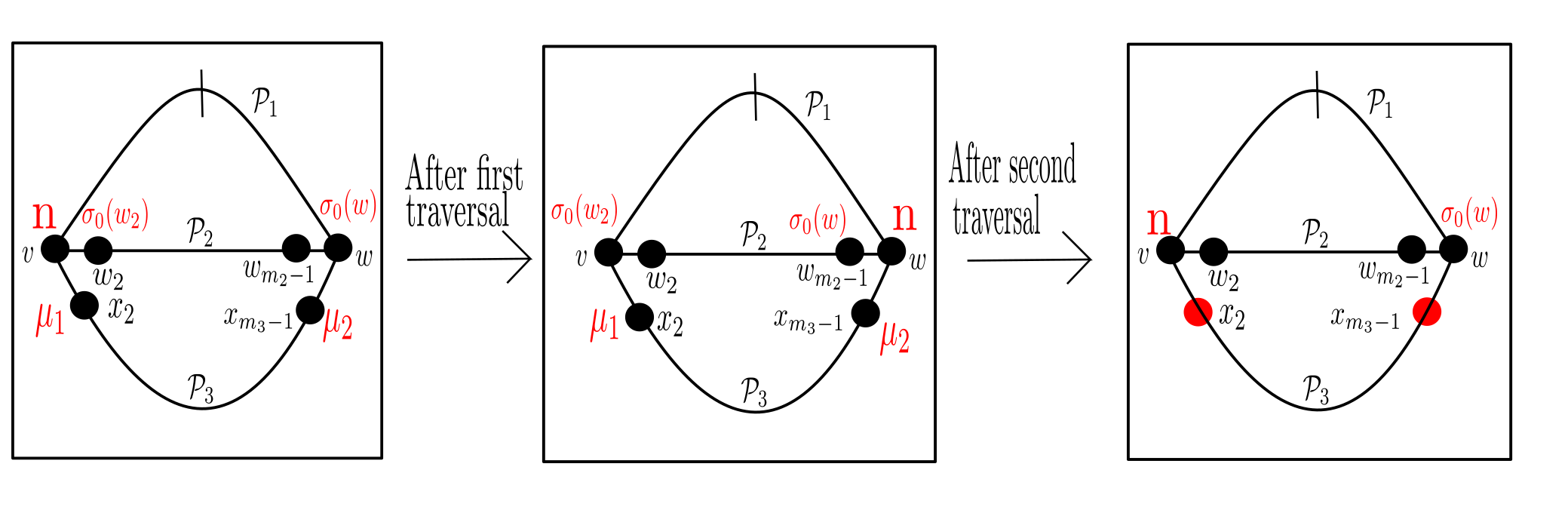}
    \caption{Illustration of the case where we assume $n_2 = 4$, after the first two traversals of $\mathcal P_2$; the hatch mark for $\mathcal P_1$ signifies that it is an edge. The red vertices in the rightmost figure indicate the possible preimages of $\sigma_0(w_2)$ following the second traversal of $\mathcal P_2$, as discussed in the text.}
    \label{fig:theta_opt_1}
\end{figure}

We can argue analogously for $n_3 = 4$, so we must have that $n_2 = n_3 = 6$, so $n_1 \leq 4$, from which it is easy to reduce this setting to $n_1 = 4$. Assume without loss of generality that $\mathcal S_{\mathcal W}'$ starts with $1$, corresponding to $\mathcal P_1$ traversed from $v$ to $w$: $\mathcal S_{\mathcal W}'$ has four $1$s, with substrings $2323\dots$ or $3232\dots$ between occurrences of $1$. Consider $\sigma_0(w)$: the path it traverses around $X$ cannot go across $\mathcal P_2$ or $\mathcal P_3$ (since $n_2 = n_3 = 6$), so it must be that the third traversal of $\mathcal P_1$ is from $w$ to $v$, swapping $\sigma_0(w)$ onto $w$ ($\sigma_0(w)$ must be swapped from $v$ to $w$ across $\mathcal P_1$ at either the third or fourth traversal of $\mathcal P_1$; if this were done in the fourth traversal, we could reorient $\mathcal S_{\mathcal W}'$ to start on this fourth traversal to derive a contradiction on $n_2 = n_3 = 6$). In particular, we can assume without loss of generality that the second traversal of $\mathcal P_1$ is also from $v$ to $w$ (if not, reorient $\mathcal S'$ to begin at what was originally the fourth traversal). For $\sigma_0(w)$ to return upon $w$ after $\mathcal S'$, it is straightforward to observe that the distance and path upon which $\sigma_0(w)$ is displaced between the first and second traversals of $\mathcal P_1$ determines the remaining trajectory by the preceding observations, as alternating traversals of $\mathcal P_1$ ``push in" the same vertex in opposite directions across $\mathcal P_1$ (see Figure \ref{fig:theta_opt_2}); this contradicts $n_2 = n_3 = 6$ in all possible settings. 

\begin{figure}[ht]
    \centering
    \includegraphics[width=0.95\textwidth]{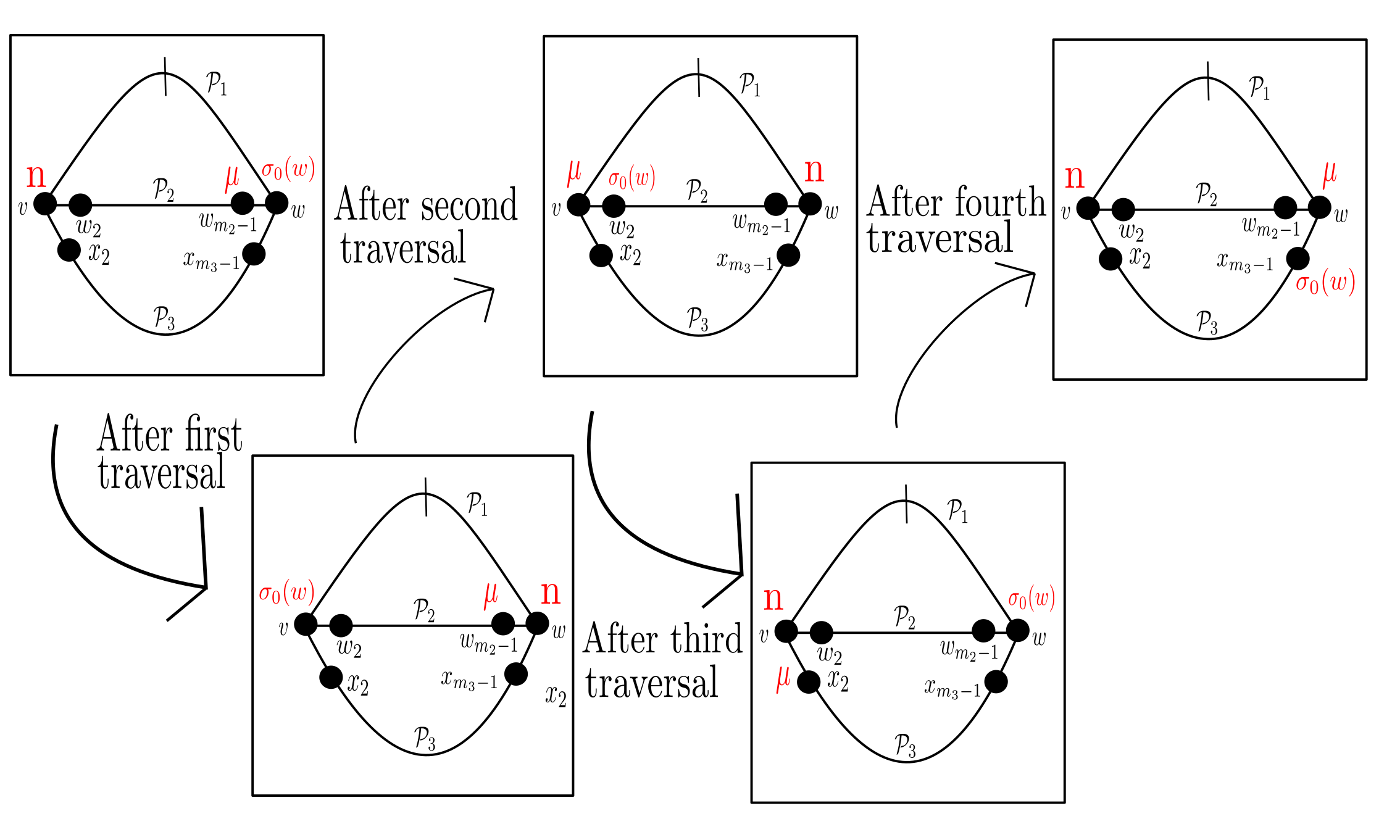}
    \caption{Case where we assume $n_1 = 4$, and $n_2 = n_3 = 6$. By the observations described in the text, the swap sequence $\mathcal S'$ is determined entirely by what happens between the first and second traversals of $\mathcal P_1$, which can then be used to derive a contradiction on $n_2 = n_3 = 6$. Recall that we can assume the first two traversals are from $v$ to $w$, and the latter two from $w$ to $v$. In the example given in this figure, $n$ swaps across $\mathcal P_2$ once prior to the second traversal of $\mathcal P_1$, so that $\mu = \sigma_0(w_{m_2-1})$ must be swapped from $v$ to $w$ on the fourth traversal. In this case, we must have $n_2 = n_3 = 3$ to return $\sigma_0(w)$ and $\mu$ to their original positions in $\sigma_0$, a contradiction. Similarly, if $d$ is the number of times we traverse $\mathcal P_2$ (assume this is traversed more times than $\mathcal P_3$ prior to the second traversal of $\mathcal P_1$) after the first traversal of $\mathcal P_1$, then it is easy to see, following the diagram above, that $n_2 = 4d-1$, yielding a contradiction for all $d \geq 2$.}
    \label{fig:theta_opt_2}
\end{figure}

This contradicts $\mathcal S'$ as a trajectory with $\phi \cdot \eta < 3(|V(\mathcal C_1)| + |V(\mathcal C_2)| + |V(\mathcal C)|)$, so for any trajectory $\mathcal S'$ achieving a cycle in $\FS(X, \Star_n)$, $\phi \cdot \eta \geq 3(|V(\mathcal C_1)| + |V(\mathcal C_2)| + |V(\mathcal C)|)$. 
\end{proof}

\noindent Call $\theta$-graphs with an edge path a $\tilde{\theta}$-graph; in general, we shall use tilde to signify the existence of an edge path in the graph. The preceding results yield the following subset of $\mathcal G$.

\begin{theorem} \label{girth_subset}
Consider the set $\tilde{\mathcal G}$ of graphs that includes the following.
\small
\begin{itemize}
    \item All cycle graphs.
    \item Barbell graphs $[\mathcal C_1, \mathcal C_2, \mathcal P]$ with $2(|V(\mathcal C_1)| + |V(\mathcal C_2)| + 2|E(\mathcal P)|) < \min\{|V(\mathcal C_1)|(|V(\mathcal C_1)|-1), |V(\mathcal C_2)|(|V(\mathcal C_2)|-1)\}$.
    \item $\theta$-graphs with $2(|V(\mathcal C_1)| + |V(\mathcal C_2)| + |V(\mathcal C)|) < \min\{|V(\mathcal C_1)|(|V(\mathcal C_1)|-1), |V(\mathcal C_2)|(|V(\mathcal C_2)|-1), |V(\mathcal C)|(|V(\mathcal C)|-1)\}$.
    \item $\tilde{\theta}$-graphs with $3(|V(\mathcal C_1)| + |V(\mathcal C_2)| + |V(\mathcal C)|) < \min\{|V(\mathcal C_1)|(|V(\mathcal C_1)|-1), |V(\mathcal C_2)|(|V(\mathcal C_2)|-1), |V(\mathcal C)|(|V(\mathcal C)|-1)\}$.
\end{itemize}
\normalsize
\noindent Then $\tilde{\mathcal G} \subset \mathcal G$, and $\tilde{\mathcal G}$ contains all cycle, barbell, and $\theta$-graphs that are in $\mathcal G$.
\end{theorem}

\subsection{Superset of $\mathcal G$}
Our main result in this section is showing that any graph in $\mathcal G$ cannot contain a proper barbell subgraph, from which we can significantly reduce the possible graphs that it contains. Towards this, we begin with some background. Recall the notion of a $\mathcal P$-induced subgraph from Definition \ref{path_induced}, for which we referred to graphs $X_{\mathcal C}$ with $\mathcal C$ a cycle in $\FS(X, \Star_n)$ that achieved its girth: we shall also refer to $Y_{\mathcal C}$, the induced subgraph of $Y=\Star_n$ with vertex set consisting of $\sigma(X_{\mathcal C})$ for any $\sigma \in V(\mathcal C)$ (i.e. all elements in $V(\Star_n)$ that are involved in the cycle $\mathcal C$ in $\FS(X, \Star_n)$).

We establish some elementary, but important observations concerning $X_{\mathcal C}$; in what follows, fix $\mathcal C$ to be a cycle subgraph in $\FS(X, \Star_n)$. We label the vertices $V(\mathcal C) = \{\sigma_i\}_{i=0}^{|V(\mathcal C)|-1}$ such that $\{\sigma_{i-1}, \sigma_i\} \in E(\mathcal C)$ for $i \in [|V(\mathcal C)|-1]$ and $\{\sigma_0, \sigma_{|V(\mathcal C)|-1}\} \in E(\mathcal C)$, and the corresponding $\mathcal C$-induced subgraph $X_\mathcal C$ of $X$. Given the path $\mathfrak n$ takes around $X_{\mathcal C}$, the following three quantities correspond to $|V(\mathcal C)|$; studying different choices will be convenient for different arguments.
\begin{enumerate}
    \item \textit{Sum of visits to each vertex.} A vertex $v \in V(X_\mathcal C)$ is said to be visited at index $i \in [\ell]$ if $\sigma_i(v) = \mathfrak n$.
    \item \textit{Sum of traversals of each edge.} An edge $\{v, w\} \in E(X_\mathcal C)$ is said to be traversed at index $i \in [\ell]$ if $\sigma_{i-1}$ and $\sigma_i$ differ on an $(X, \Star_n)$-friendly swap over $\{v, w\}$. (For $i = 1$, $i+1$ corresponds to $\ell$.)
    \item \textit{Sum of swaps with non-central vertices of $\Star_n$ in $Y_\mathcal C$.} A vertex $v \in V(Y_\mathcal C) \setminus \{\mathfrak n\}$ is said to be swapped at index $i \in [\ell]$ if $\{\sigma_{i-1}, \sigma_i\} \in E(\mathcal C)$ corresponds to an $(X, \Star_n)$-friendly swap by swapping $\mathfrak n$ with $v$.
\end{enumerate}
From here, we observe that $|V(\mathcal C)| \geq 2|V(X_\mathcal C)|$, since every vertex $v \in V(X_\mathcal C)$ must be visited at least twice (to swap out and back in the vertex originally on $v$), and the sum of the number of visits to each vertex is equal to $|V(\mathcal C)|$. Recall from Remark \ref{girth_basics} that the path taken by any $y \in V(\Star_n)$ during $\sigma$ either ``folds in on itself" or yields a simple cycle: in this latter setting, we shall denote this simple cycle by $y_\mathcal C$.

Corresponding to the third quantity above, let $(v_1, \dots, v_{|V(Y_\mathcal C)|-1})$ be some enumeration of $V(Y_\mathcal C) \setminus \{\mathfrak n\}$, and take $(y_1, \dots, y_{|V(Y_\mathcal C)|-1}) \in \mathbb N^{|V(Y_\mathcal C)|-1}$ so that for $i \in [|V(Y_\mathcal C)|-1]$, $y_i$ denotes the number of swaps $\mathfrak n$ makes with $v_i$ over $\mathcal C$. Then $y_i \geq 2$ for all $i \in [|V(Y_\mathcal C)|-1]$ (every non-central vertex must be swapped with $\mathfrak n$ at least twice: once to swap it out of its original position, and once to swap it back), and $|V(\mathcal C)| = \sum_{i=1}^{|V(Y_\mathcal C)|-1} y_i \geq 2(|V(Y_\mathcal C)|-1) + 2$. In particular, we cannot have $y_i = 2$ for all $i \in [|V(Y_\mathcal C)|-1]$, so there necessarily exists $v_i \in V(Y_\mathcal C) \setminus \{\mathfrak n\}$ swapped strictly more than twice. We denote the corresponding vector of swap counts by $y_{\mathcal C} = (y_1, \dots, y_{|V(Y_\mathcal C)|-1})$, and its elementwise sum by $||y_{\mathcal C}||_1$. 

We finally remark that since $\FS(X, Y)$ is bipartite, if we want to deduce an upper bound of the form $|V(\mathcal C)| \leq 2k$, it suffices to show that $|V(\mathcal C)| < 2(k+1)$. Similarly, we can deduce a lower bound $|V(\mathcal C)| \geq 2(k+1)$ by showing $|V(\mathcal C)| > 2k$.

\begin{theorem} \label{barbell_proof}
Let $\mathfrak C$ be the set of all cycle subgraphs in $\FS(X, \Star_n)$ that achieve its girth, and for $\mathcal C \in \mathfrak C$, let $X_{\mathcal C}$ be the $\mathcal C$-induced subgraph of $X$. Denote the set of all such subgraphs as $\mathcal X = \{X_{\mathcal C}\}_{\mathcal C \in \mathfrak C}$, and say there exists some $X_{\mathcal C} \in \mathcal X$ with a barbell subgraph. Then there exists a cycle or barbell graph in $\mathcal X$.
\end{theorem}

\begin{proof}
We shall proceed by contradiction, for which we assume that all barbell subgraphs of graphs in $\mathcal X$ are proper, and there do not exist cycles in $\mathcal X$. Take $X_\mathcal C \in \mathcal X$ containing a proper barbell subgraph $\mathcal B$ that has decomposition $[\mathcal C_1, \mathcal C_2, \mathcal P]$. We must have $|V(\mathcal C_1)|, |V(\mathcal C_2)| \geq 5$, as $\mathcal B$ has at least $2k-1$ vertices (letting $k$ be the size of the smaller cycle in the barbell) so $|V(\mathcal C)| \geq 4k-2$. Since $k(k-1) \leq 4k-2$ for $k \geq 4$, if either $\mathcal C_1$ or $\mathcal C_2$ had at most $4$ vertices, we could strictly traverse the cycle subgraph, getting a smaller cycle than $\mathcal C$ in $\FS(X, \Star_n)$. In the case that $\mathcal P$ has at least one edge, we can assume $|V(\mathcal C_1)|, |V(\mathcal C_2)| \geq 6$ by similar reasoning.

We shall broadly argue as follows. Take a barbell subgraph $\mathcal B$ of $X_\mathcal C$ such that $\mathcal P$ has smallest possible length: we show that the swap sequence taken by $\mathfrak n$ along $X_\mathcal C$ to achieve $\mathcal C$ cannot ``improve" the swap sequence along $\mathcal B$ given by Proposition \ref{barbell_bd}, which we shall say has corresponding cycle subgraph in $\FS(X, \Star_n)$ denoted by $\mathcal C_\mathcal B$ (i.e. we show that $|V(\mathcal C)| \geq |V(\mathcal C_\mathcal B)|$). This contradicts either $|V(\mathcal C)| = g(\FS(X, \Star_n))$ (if $|V(\mathcal C)| > |V(\mathcal C_\mathcal B)|$) or the noninclusion of $\mathcal B$ in $\mathcal X$ (if $|V(\mathcal C)| = |V(\mathcal C_\mathcal B)|$). In the following argument, we handle small cases for the length of $\mathcal P$ (the number of edges of $\mathcal P$) before we argue generally on all $\mathcal P$ with sufficiently large length. We shall also frequently refer to the swap sequences $\mathcal S$ and $\mathcal S_\mathcal B$ corresponding to the cycle subgraphs $\mathcal C$ and $\mathcal C_\mathcal B$ of $\FS(X, \Star_n)$, respectively (see the footnote under Remark \ref{girth_basics}), as it will occasionally be more convenient to refer directly to the swaps applied in the cycle $\mathcal C$.

\paragraph{$\mathcal P$ has length $0$.}

Here, $\mathcal C_1$ and $\mathcal C_2$ share exactly one vertex, and $|V(\mathcal C_{\mathcal B})| = 2(|V(\mathcal B)|+1)$. If $\mathcal B$ were a proper subgraph of $X_\mathcal C$, then certainly $V(\mathcal B) = V(X_\mathcal C)$; if this were not true, we would have $|V(\mathcal C)| \geq 2|V(X_\mathcal C)| \geq 2(|V(\mathcal B)|+1)$, contradicting the construction of $\mathcal X$ either by the exclusion of $\mathcal B$ or by $|V(\mathcal C)| > |V(\mathcal C_\mathcal B)|$.) Hence, $X_\mathcal C$ can only add edges to $\mathcal B$, so that $V(X_\mathcal C) = V(\mathcal B)$.

Since $|V(\mathcal C_\mathcal B)| = 2(|V(\mathcal B)|+1) > |V(\mathcal C)| \geq 2|V(\mathcal B)|$, for $\mathcal S$ to improve $\mathcal S_\mathcal B$, $\mathfrak n$ must visit the vertex $v \in V(X_\mathcal C)$ of degree $4$ precisely twice; assume this is done over $\mathcal S$. Assume (without loss of generality) that over $\mathcal S$, $\mathfrak n$ does not start on $v$ or a neighbor of $v$, and that $a \in V(Y_\mathcal C)$ is originally on $v$. The two visits to $v$, which must involve all four edges incident to $v$, necessarily traverses each incident edge exactly once. In order to finish at $v$, $a$ must have taken some other sequence of swaps between the first and second visits to $v$. Thus, we can extract a cycle subgraph $a_\mathcal C$ with length at least $5$, so the corresponding entry in $y_\mathcal C$ is at least $5$. But then $|V(\mathcal C)| = ||y_\mathcal C||_1 \geq ||y_{\mathcal C_\mathcal B}||_1 = |V(\mathcal C_\mathcal B)|$, a contradiction, so $\mathcal B$ cannot be a proper subgraph of $X_\mathcal C$.

\paragraph{$\mathcal P$ has length $1$.}
As remarked, we have (for this and all following cases) $|V(\mathcal C_1)| \geq 6, |V(\mathcal C_2)| \geq 6$. Assume that we choose $\mathcal B$ such that $|V(\mathcal C_1)| + |V(\mathcal C_2)|$ is minimal. Here, $\mathcal P$ is an edge, and $\mathcal S_\mathcal B$ consists of $2|V(\mathcal B)| + 4$ swaps (the two degree $3$ vertices are visited precisely four times, while all other vertices are visited twice). When constructing $X_C$ from its proper subgraph $\mathcal B$, either $V(X_\mathcal C) = V(\mathcal B)$, or $V(X_\mathcal C)$ has one extra vertex without causing the inequality chain $|V(\mathcal C)| \geq 2|V(X_\mathcal C)| \geq 2|V(\mathcal B)| + 4$ or contradicting the strict improvement of $\mathcal S$ over $\mathcal S_\mathcal B$. We break into cases based on these possibilities. 

\subparagraph{Case 1: $V(X_\mathcal C) = V(\mathcal B)$.}

Here, we can only add edges to $\mathcal B$ to achieve $X_\mathcal C$. By assumption on $\mathcal B$ having minimal $|V(\mathcal C_1)| + |V(\mathcal C_2)|$, edges can only be added between $V(\mathcal C_1)$ and $V(\mathcal C_2)$, and there must be at least one additional edge, since $\mathcal B$ is proper in $X_\mathcal C$. Additional edges must be between two vertices of degree $2$ in $\mathcal B$; if we added an edge to a vertex with degree $3$ and a vertex of degree $2$ on the other cycle, there would exist a barbell subgraph of $X_\mathcal C$ with connecting path of length $0$. Hence, there are at least four vertices with degree $3$ in $V(X_\mathcal C)$. The sequence $\mathcal S_\mathcal B$ involves $2|V(\mathcal B)| + 4$ swaps: for $2|V(\mathcal B)| \leq |V(\mathcal C)| < 2|V(\mathcal B)| + 4$ (for $\mathcal S$ to strictly improve $\mathcal S_\mathcal B$), there can exist at most two vertices visited more than twice via $\mathcal S$ on $X_\mathcal C$. In particular, there must exist at least two vertices incident to edges between vertices of $\mathcal C_1$ and $\mathcal C_2$ that are visited exactly twice. Denote these by $v_1$ and $v_2$, and assume without loss of generality that $n \in V(\Star_n)$ starts on neither $v_1$, $v_2$, nor any vertex incident to them as it traverses $X_\mathcal C$.

Let $a \in V(Y_\mathcal C)$ be the vertex initially on $v_1$, and $b \in V(Y_\mathcal C)$ denote the value swapped onto $v_1$ after the first visit. The two visits of $\mathfrak n$ onto $v_1$ traverse all incident edges to $v_1$, so $b$ returns to a different position after the second visit to $v_1$. Hence, $b_C$ yields a cycle of length at least $6$ in $X_\mathcal C$ (anything smaller would give a cycle in $\FS(X, \Star_n)$ smaller than the lower bound of $2|V(X_\mathcal C)| \geq 24$ on the girth $|V(\mathcal C)|$). Thus, $b$ must be swapped at least $6$ times. Similarly, consider the vertex $v_2$, for which there exists some $b' \in V(\Star_n)$ so that $b'_\mathcal C$ yields a cycle of length at least $6$ in $X_\mathcal C$ (possibly the same as that from $b_\mathcal C$). If $b \neq b'$, there exist at least two vertices of $\Star_n$ which traverse a cycle of length at least $6$, so that $|V(\mathcal C)| = ||y_\mathcal C||_1 > ||y_{\mathcal C_\mathcal B}||_1 = 2|V(\mathcal B)| + 4 = |V(\mathcal C_\mathcal B)|$, a contradiction on $\mathcal C$ achieving the girth of $\FS(X, \Star_n)$. This inequality also must hold whenever $b = b'$, but with $b$ swapped strictly more than $6$ times.

\medskip

It remains to be shown that if $b = b'$ and $b$ is swapped exactly $6$ times, then $\mathcal S$ cannot improve $\mathcal S_\mathcal B$. Here, we must have that $b_\mathcal C$ is a cycle of length $6$. We also must have $|V(\mathcal B)| \leq 14$ (if not, i.e. $|V(\mathcal B)| \geq 15$, we could strictly traverse $b_\mathcal C$ for a cycle in $\FS(X, \Star_n)$ with size at most $|V(\mathcal C)| \geq 2|V(\mathcal B)| \geq 30$). Note that $\mathcal S$ cannot achieve a number of swaps equal to the lower bound $2|V(\mathcal B)|$ since $b$ swapped $6$ times yields $||y_\mathcal C||_1 \geq 2|V(\mathcal B)| + 2$, so this can be strengthened to $|V(\mathcal B)| \leq 13$ (if $|V(\mathcal B)| = 14$, then $||y_\mathcal C||_1 \geq 30$, so we can traverse $b_\mathcal C$ for a cycle in $\FS(X, \Star_n)$ achieving the girth by Corollary \ref{quad_bd}). Recalling that $|V(\mathcal C_1)|, |V(\mathcal C_2)| \geq 6$, the only possibilities for $\mathcal C_1$ and $\mathcal C_2$ are that they are both $6$-cycles, or one is a $6$-cycle and the other a $7$-cycle. We can also assume that any $y \in V(Y_\mathcal C) \setminus \{b, n\}$ is swapped exactly twice, since we would otherwise have that $||y_\mathcal C||_1 \geq 2|V(\mathcal B)|+4$, contradicting either $|V(\mathcal C)| = g(\FS(X, \Star_n))$ (if a strict inequality) or $\mathcal B \notin \mathcal X$ (if an equality), as $\mathcal S_\mathcal B$ here yields a cycle in $\FS(X, \Star_n)$ with size $2|V(\mathcal B)| + 4$.

Observe that no vertex in $V(X_\mathcal C)$ has degree greater than $3$, as this would yield the existence of a barbell subgraph of $X_\mathcal C$ with length $0$ connecting path (recall that edges can only be added between $V(\mathcal C_1)$ and $V(\mathcal C_2)$). Any $y \in V(Y_C) \setminus \{b, n\}$ is swapped exactly twice, so $y$ must traverse a specific edge in $E(X_\mathcal C)$ twice, and these are the only traversals of this edge (an additional traversal would require swapping $y$ more than twice). These edges must be disjoint from $E(b_\mathcal C)$ (as $b$ is, at some point, swapped along every edge in $E(b_\mathcal C)$), so considering the sum of edge traversals yields $|V(\mathcal C)| \geq 2(|V(\mathcal B)|-2) + 6 = 2(|V(\mathcal B)|+1)$ (the expression $|V(\mathcal B)|-2$ is from the edges corresponding to vertices in $V(Y_\mathcal C) \setminus \{b, n\}$, and the additional $6$ from considering the edges of the $6$-cycle $b_C$); in particular, note that $|E(X_\mathcal C)| \geq |V(\mathcal B)| + 4$. Enumerate $V(b_\mathcal C)$ by $\{1, \dots, 6\}$: any $y \in V(Y_\mathcal C) \setminus \{b, n\}$ is swapped exactly twice along a particular edge, from which it follows that all vertices in $b_\mathcal C$ must have degree\footnote{Specifically, assume without loss of generality that $\mathfrak n$ is not initially placed on $b_\mathcal C$. The vertices initially upon $\{1, \dots, 6\}$ in $V(Y_\mathcal C) \setminus \{b\}$ must traverse an edge outside the cycle $b_\mathcal C$, as every such edge in $E(b_\mathcal C)$ is used for a swap between $b$ and $\mathfrak n$ at some point. Furthermore, if $b$ starts on a vertex of $V(b_\mathcal C)$, then the swap after the first swap between $\mathfrak n$ and $b$ must leave $b_\mathcal C$, so this vertex must also have degree $3$.} $3$, and have third incident vertices all distinct since a cycle with size at most $5$ cannot exist in $X_\mathcal C$: label these vertices $\{1', \dots, 6'\}$. Figure \ref{fig:b_swap_6} depicts this subgraph of $X_\mathcal C$.

\begin{figure}[ht]
    \centering
    \includegraphics[width=0.35\textwidth]{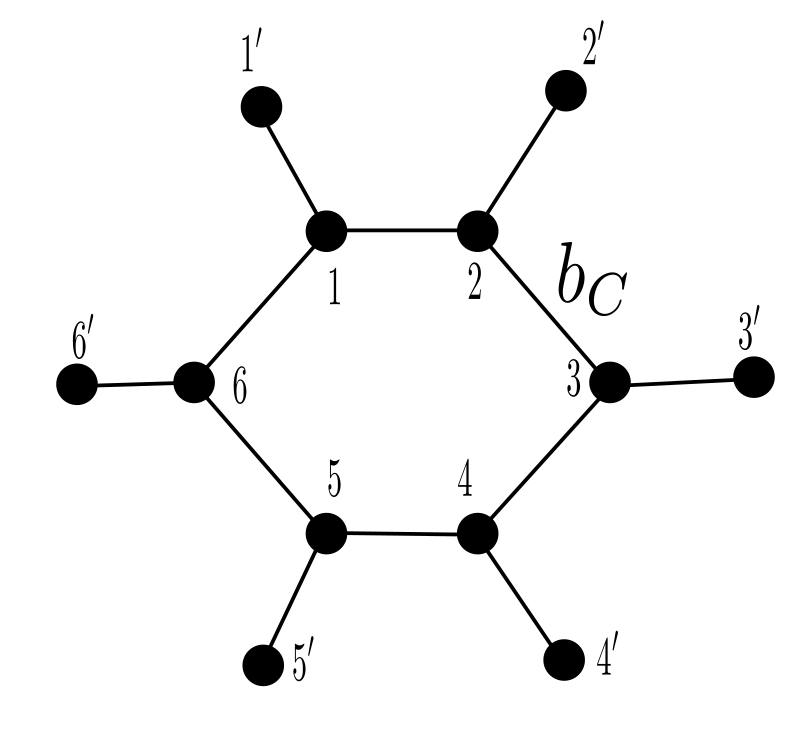}
    \caption{The setup for the case where $b=b'$ and $b$ is swapped exactly $6$ times. The following depicts what must be isomorphic to a subgraph of $X_\mathcal C$ under this setting, as discussed in the preceding text. Considering the two possibilities $|V(\mathcal C_1)| = |V(\mathcal C_2)| = 6$ and (without loss of generality) $|V(\mathcal C_1)| = 6$, $|V(\mathcal C_2)| = 7$, arguing on this subgraph can now be used to derive a contradiction.}
    \label{fig:b_swap_6}
\end{figure}

Assume without loss of generality that the first swap is that moving $b$ along the edge $\{1,2\}$, so the swap sequence $\sigma$ must be of the following form, where each of the explicitly labeled swaps corresponds to those that swap the vertex $b$ along the cycle subgraph $b_\mathcal C$ (e.g. the term $(1\ 2)$ corresponds to swapping $b$ from $1$ to $2$ with $n \in V(\Star_n)$). In particular, $n \in V(\Star_n)$ begins its traversal on vertex $2$, and the terms $s_i$ for $i \in [6]$ represent subsequences of swaps within $\mathcal S$.
\begin{align*}
    \mathcal S = (1 \ 2)s_1(2\ 3)s_2 (3\ 4)s_3 (4\ 5)s_4 (5\ 6)s_5 (6\ 1)s_6 
\end{align*}
Observe that each subsequence $s_i$ for $i \in [5]$ must have length at least $4$ (i.e. $s_i$ must involve at least $4$ $(X_\mathcal C, Y_\mathcal C)$-friendly swaps). As an example, consider the trajectory that $\mathfrak n$ takes in $s_1$: from $(1\ 2)$ to $(2\ 3)$, $n \in V(\Star_n)$ must traverse a path from vertex $1$ to vertex $3$, and any trajectory with length less than $4$ would yield the existence of a cycle of length at most $5$ in $X_C$, a contradiction on $|V(\mathcal C)| = g(\FS(X, \Star_n))$. This causes the number of swaps in $\mathcal S$ up to the swap $(6\ 1)$ to be at least $5\cdot 4 + 6 = 26$, and right after the swap $(6\ 1)$, $n \in V(\Star_n)$ lies upon vertex $6$. We cannot return $\mathfrak n$ to its starting point of vertex $2$, as this requires $s_6$ to have length at least $3$, so $|V(\mathcal C)| \geq 30$, a contradiction either on $|V(\mathcal C)| = g(\FS(X, \Star_n))$ (if the inequality is strict) or $b_\mathcal C \notin \mathcal X$ (if an equality) since we can traverse $b_\mathcal C$ for a $30$-cycle by Corollary \ref{quad_bd}.

In all possible cases, we deduce a contradiction on the claim that $\mathcal S$ improves $\mathcal S_\mathcal B$.

\subparagraph{Case 2: $|V(X_\mathcal C)| = |V(\mathcal B)| + 1$ adds one vertex to $\mathcal B$.}
    
The additional vertex must have degree at least $2$ in $X_\mathcal C$. Observe that any such vertex cannot be adjacent to two vertices in $\mathcal C_1$ or two vertices in $\mathcal C_2$ by choice of $\mathcal B$ (minimality of $|V(\mathcal C_1)| + |V(\mathcal C_2)|$) and the fact that $3$-cycles and $4$-cycles cannot be subgraphs of $X_\mathcal C$ without contradicting optimality of $\sigma$. Hence, this new vertex must have degree $2$, and is adjacent to a vertex in $\mathcal C_1$ and $\mathcal C_2$. To improve $\mathcal S_\mathcal B$, which has $2|V(\mathcal B)| + 4$ swaps, we must have that $\sigma$ is a sequence of swaps such that $n \in V(\Star_n)$ visits all vertices of $X_\mathcal C$ exactly twice, since $|V(X_\mathcal C)| = 2|V(\mathcal B)| + 2$. As before, consider any vertex of degree $3$ in $X_\mathcal C$, for which we can similarly argue that there must exist some vertex $y \in V(Y_\mathcal C)$ which traverses a $6$-cycle, and thus $|V(\mathcal C)| = ||y_\mathcal C||_1 > 2|V(X_\mathcal C)|$, a contradiction on $\mathcal S$ improving $\mathcal S_\mathcal B$.

\paragraph{$\mathcal P$ has length at least $2$.}

$\mathcal S_\mathcal B$ is a sequence of $2(|V(\mathcal C_1)| + |V(\mathcal C_2)| + 2d+2)$ swaps, where $d$ is the number of inner vertices of $\mathcal P$ (and we have $d \geq 1$). Hence, we have that $|V(\mathcal C)| < 2(|V(\mathcal C_1)| + |V(\mathcal C_2)| + 2d+2)$ (this is strict, as we assumed $\mathcal B$ does not lie in $\mathcal X$), or $|V(\mathcal C)| \leq 2(|V(\mathcal C_1)| + |V(\mathcal C_2)| + 2d+1)$. 
    
Begin by assuming the existence of another path $\mathcal P' = \{v, v_1, \dots, v_k, w\}$, with $v \in V(\mathcal C_1)$, $w \in V(\mathcal C_2)$, and $v_i \notin V(\mathcal C_1) \cup V(\mathcal C_2)$ for all $i \in [k]$. Either $k=d$ (by assumption on $\mathcal P$ as the shortest connecting path) or $k=d+1$ (to respect the upper bound on $|V(\mathcal C)|$ given above, as every vertex must be visited at least twice): it follows that there can be at most one other path $\mathcal P'$. Additionally, $\mathcal P'$ must be vertex-disjoint from $\mathcal P$: if not, we could have constructed a barbell subgraph of $X_\mathcal C$ with a strictly shorter connecting path. We break into cases based on whether $k=d$ or $k=d+1$.
    
\subparagraph{$\mathbf{k=d}$.}

Recall that we have $|V(\mathcal C_1)| \geq 6$ and $|V(\mathcal C_2)| \geq 6$. In the case that $d = 1$, the two additional vertices yield $|V(X_\mathcal C)| \geq 14$, so that $|V(C)| \geq 2|V(X_\mathcal C)| \geq 28$. To achieve $|V(\mathcal C)| = 28$, we must have that $|V(X_\mathcal C)| = 14$, and every vertex of $X_\mathcal C$ must be visited precisely twice, in which case we can consider a vertex with degree at least $3$ to obtain $y \in V(Y_\mathcal C)$ such that $y_\mathcal C$ is a cycle of size at least $6$, so that $|V(\mathcal C)| = ||y_\sigma||_1 > 28$. Thus, $|V(\mathcal C)| \geq 30$, so there cannot exist $6$-cycles in $X_\mathcal C$, as we can achieve a cycle in $\FS(X, \Star_n)$ with at most as many swaps as $\mathcal S$ by strictly traversing any $6$-cycle. 

Henceforth, we can assume $|V(\mathcal C_1)| \geq 7, |V(\mathcal C_2)| \geq 7$. We have  $|V(\mathcal C)| \geq 2(|V(\mathcal C_1)| + |V(\mathcal C_2)| + 2d)$, so for $\mathcal S$ to improve $\mathcal S_\mathcal B$, there exist at most two vertices that are visited more than twice. In particular, some vertex with degree at least $3$ must be visited exactly twice. Thus, by arguing as before, we have the existence of some $y \in V(Y_\mathcal C)$ such that $y_\mathcal C$ is a cycle of size at least $7$, so that $|V(\mathcal C)| = ||y_\mathcal C||_1 > 2(|V(\mathcal C_1)| + |V(\mathcal C_2)| + 2d + 1)$, a contradiction.

\subparagraph{$\mathbf{k=d+1}$.}
    
The only way that $\mathcal S$ can improve $\mathcal S_\mathcal B$ here is if $\mathcal S$ visits each vertex of $X_\mathcal C$ exactly twice, as $X_\mathcal C$ has exactly $|V(\mathcal C_1)| + |V(\mathcal C_2)| + 2d + 1$ vertices. As before, considering any vertex of degree $3$ in $X_\mathcal C$ yields the existence of some $y \in Y_\mathcal C$ such that $y_\mathcal C$ is a cycle of size at least $7$, so that $|V(\mathcal C)| = ||y_\mathcal C||_1 \geq 2(|V(\mathcal C_1)| + |V(\mathcal C_2)| + 2d + 1)$, a contradiction.

\medskip
    
Now assume that no such path $\mathcal P'$ exists in $X_\mathcal C$. Here, there cannot exist a path from an inner vertex of $\mathcal P$ to vertices of $\mathcal C_1$, $\mathcal C_2$, or $\mathcal P$ itself with inner vertices disjoint from $V(\mathcal C_1)$, $V(\mathcal C_2)$ and $V(\mathcal P)$, as any such path would contradict the minimality of the length of $\mathcal P$ in the barbell subgraph $\mathcal B$. Denote the path $\mathcal P = \{v, v_1, \dots, v_d, w\}$ with $v \in V(\mathcal C_1)$ and $w \in V(\mathcal C_2)$, and construct the subgraph $\mathcal C_1^*$ of $X_C$ as follows. Setting $\mathcal C_1^{(0)} = \mathcal C_1$, $\mathcal C_1^{(1)}$ is achieved by appending to $\mathcal C_1^{(0)}$ all edges and vertices incident to vertices in $V(\mathcal C_1^{(0)})$, excluding vertex $v_1$ and edge $\{v, v_1\}$. Generally, construct $\mathcal C_1^{(i)}$ by appending to $\mathcal C_1^{(i-1)}$ all edges and vertices incident to vertices in $V(\mathcal C_1^{(i-1)}) \setminus V(\mathcal C_1^{(i-2)})$. Continue until the process terminates to yield $C_1^*$. Similarly construct the subgraphs $\mathcal C_2^*$ (excluding vertex $v_d$ and edge $\{v_d, w\}$ for $C_2^{(1)}$) and $\mathcal P^*$ (letting $\mathcal P^{(0)}$ be the subpath consisting of the inner vertices $\{v_1, \dots, v_d\}$ of the path $\mathcal P$, and excluding $v$ and $w$ in $\mathcal P^{(1)}$). It is not hard to show that for any $v \in V(\mathcal C_1^*)$, the smallest $i$ such that $v \in V(\mathcal C_1^{(i)})$ denotes the length of a shortest path from a vertex in $V(\mathcal C_1)$ to $v$, and that in such a shortest path the $j$th inner vertex first appears in the set $V(\mathcal C_1^{(j)})$. An analogous statement holds for vertices in $V(\mathcal C_2^*)$ and $V(\mathcal P^*)$.

We can now decompose $X_\mathcal C$ by $[\mathcal C_1^*, \mathcal C_2^*, \mathcal P^*]$, which must have that $V(X_\mathcal C) = V(\mathcal C_1^*) \sqcup V(\mathcal C_2^*) \sqcup V(\mathcal P^*)$ and $E(X_C) = E(\mathcal C_1^*) \sqcup E(\mathcal C_2^*) \sqcup E(\mathcal P^*) \sqcup \{\{v, v_1\}, \{v_d, w\}\}$. A violation of one of these statements would contradict either our choice of $\mathcal P$ as the shortest connecting path for a barbell in $X_\mathcal C$ or the nonexistence of a second path $\mathcal P'$ between a vertex of $\mathcal C_1$ and a vertex of $\mathcal C_2$. It is not hard to see that the sets $V(\mathcal C_1^*), V(\mathcal C_2^*), V(\mathcal P^*)$ include all vertices of $V(X_\mathcal C)$, as $X_\mathcal C$ is connected. Say we take $\mathcal C_1^*$, which certainly has every set $V(\mathcal C_1^{(i)})$ (and thus the entirety of $V(\mathcal C_1^*)$) disjoint from $V(\mathcal C_2)$ and the inner vertices of $\mathcal P$. Now inductively consider $\mathcal P^*$: for $i \geq 1$, assuming $V(\mathcal P^{(i-1)}) \cap V(\mathcal C_1^*) = \emptyset$, if there existed $v \in V(\mathcal P^{(i)}) \cap V(\mathcal C_1^*)$, take smallest $j$ such that there exists $v \in V(\mathcal P^{(i)}) \cap V(\mathcal C_1^{(j)})$, yielding a path between $V(\mathcal C_1)$ and an inner vertex of $\mathcal P$ by ``tracing back" the constructions with inner vertices disjoint from $V(\mathcal B)$, a contradiction. We similarly conclude that $V(\mathcal C_2^*) \cap V(\mathcal C_1^*) = V(\mathcal C_2^*) \cap V(\mathcal P^*) = \emptyset$ by arguing inductively on the construction of $\mathcal C_2^*$. For the claim on the decomposition of $E(X_\mathcal C)$, the disjointedness of the proposed edge sets follows from the analogous statement on the vertices $V(X_\mathcal C)$. To show the union, consider arbitrary $e = \{a, b\} \in E(X_\mathcal C)$ not either $\{v, v_1\}$ or $\{v_d, w\}$, and consider vertex $a$: if $a \in V(\mathcal C_1^*)$, then we must have $b \in V(\mathcal C_1^*)$ and $e \in E(\mathcal C_1^*)$ by construction of $\mathcal C_1^*$. A similar argument holds if we initially assume $a \in V(\mathcal C_2^*)$ or $a \in V(\mathcal P^*)$, so every $e \in E(X_\mathcal C)$ lies in $E(\mathcal C_1^*), E(\mathcal C_2^*), E(\mathcal P^*)$, or $\{\{v, v_1\}, \{v_d, w\}\}$.

It also follows from this decomposition that any path from a vertex in $V(\mathcal C_1^*)$ to a vertex in $V(\mathcal C_2^*)$ necessarily involves all vertices of $\mathcal P$ at some point, in order respecting that of $\mathcal P$. Specifically, let $\Tilde{\mathcal P} = \{\nu_0, \nu_1, \nu_2, \dots, \nu_k\}$ with $\nu_0 \in V(\mathcal C_1^*)$ and $\nu_k \in V(\mathcal C_2^*)$ be such a path. Certainly, we have that $v \in \Tilde{\mathcal P}$ and is the earliest vertex of $\mathcal P$ in $\Tilde{\mathcal P}$, as any edge traversed from $\nu_0 \in V(\mathcal C_1^*)$ must either be in $E(\mathcal C_1^*)$ or be equal to $\{v, v_1\}$. Now assume the vertices $\{v, v_1, \dots, v_i\}$ have appeared in $\mathcal P$, following this order. Starting from an occurrence of the vertex $v_i$, $\Tilde{\mathcal P}$ pursues a sequence of edges that either immediately goes to $v_{i-1}$ or $v_{i+1}$ (i.e. the next edge is $\{v_{i-1}, v_i\}$ or $\{v_i, v_{i+1}\}$) or eventually returns to $v_i$ after traversing vertices in $V(\mathcal P^*) \setminus V(\mathcal P^{(0)})$ (returning to a different vertex in $\mathcal P^{(0)}$ first would give the existence of a barbell subgraph with shorter connecting path, a contradiction on the choice of $\mathcal B$). In particular, the next new vertex of $\mathcal P$ that appears in $\Tilde{\mathcal P}$ exists, and must necessarily be $v_{i+1}$.

\begin{figure}[ht]
    \centering
    \includegraphics[width=0.6\textwidth]{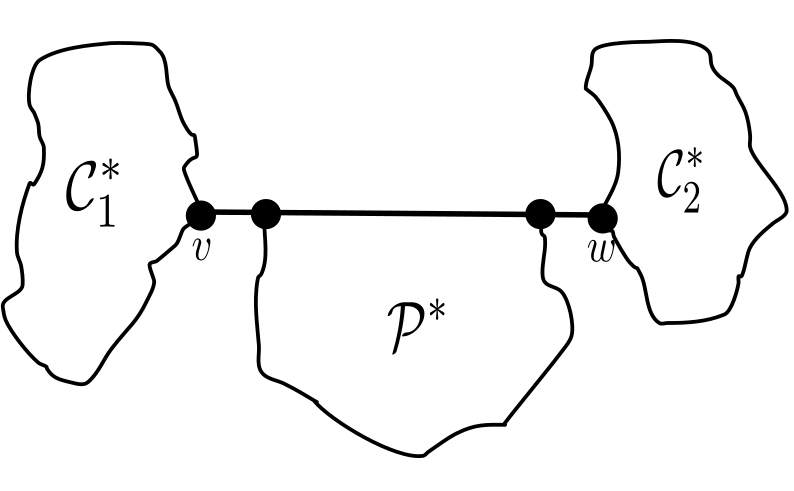}
    \caption{Schematic diagram illustrating this decomposition. Since $\mathcal P$ is the shortest path in any barbell subgraph of $X_\mathcal C$ and there does not exist another path $\mathcal P'$ connecting vertices in $V(\mathcal C_1)$ and $V(\mathcal C_2)$, $X_\mathcal C$ must possess this structure. Indeed, if the vertex sets of $\mathcal C_1^*$, $\mathcal C_2^*$, and $\mathcal P^*$ were not pairwise disjoint, we would raise a contradiction on either the minimality of the length of $\mathcal P$ in a barbell subgraph of $X_\mathcal C$ or the nonexistence of a second such path $\mathcal P'$. From here, it should be clear that every vertex in $\mathcal P$ must be visited at least four times in any optimal swap sequence achieving a cycle in $\FS(X_\mathcal C, \Star_n)$.}
    \label{fig:4.15_decomp}
\end{figure}

Here, we must necessarily visit each vertex in $\mathcal P$ at least four times via the swap sequence $\mathcal S$, which will show that $\mathcal S_\mathcal B$ cannot be improved. Indeed, assume the contrary: certainly the path $\mathcal P$ is crossed an even number of times (to return to a starting vertex in $\mathcal C_1^*$, where we shall assume without loss of generality $n \in V(\Star_n)$ begins its traversal), so $\mathcal P$ is crossed exactly twice. But then $\mathcal S$ must have completed some cycle in $\FS(X, \Star_n)$ strictly on the subgraph $V(\mathcal C_2^*)$ as we do not return there, a contradiction on $\mathcal S$ as a sequence of swaps achieving the girth of $\FS(X, \Star_n)$, as there exists a strictly improved sequence of swaps that is a proper contiguous subset of the sequence of swaps given by $\mathcal S$ that yields a cycle in $\FS(X, \Star_n)$, a contradiction on the optimality of $\mathcal S$.
\end{proof}

\noindent From the discussion following Problem \ref{girth_problem} concerning the graphs in $\mathcal G$ and Theorem \ref{barbell_proof}, we can deduce that any graph in $\mathcal G$ does not have a proper barbell subgraph, which significantly improves our understanding of the possible graphs that lie in $\mathcal G$, and thus the possible trajectories $\mathfrak n$ takes around a graph $X$ to achieve the girth of $\FS(X, \Star_n)$.

\begin{proposition} \label{remaining_cases}
Let $\mathcal G'$ include the set of graphs $\tilde{\mathcal G}$ from Theorem \ref{girth_subset}, as well as all instances of $\theta_4$, $\theta_5$, $\tilde{\theta_6}$, $K_4^*$, and $K_{3,3}^*$-graphs, as depicted in Figure \ref{fig:rem_cases}. Then $\mathcal G \subset \mathcal G'$.
\end{proposition}

\begin{figure}[ht]
    \centering
    \begin{minipage}{.3\linewidth}
    \centering
    \subfloat[$\theta_4$-graphs.]{\label{fig:theta4}\includegraphics[width=\textwidth]{figures/rem_cases_a.png}}
    \end{minipage}%
    \hfill
    \begin{minipage}{.3\linewidth}
    \centering
    \subfloat[$\theta_5$-graphs.]{\label{fig:theta5}\includegraphics[width=\textwidth]{figures/rem_cases_b.png}}
    \end{minipage}%
    \hfill
    \begin{minipage}{.3\linewidth}
    \centering
    \subfloat[$\tilde{\theta_6}$ graphs.]{\label{fig:tilde_theta6}\includegraphics[width=\textwidth]{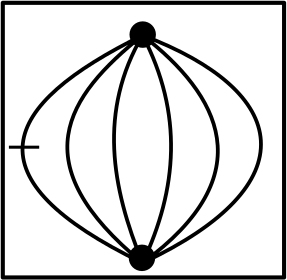}}
    \end{minipage} \\
    \begin{minipage}{.48\linewidth}
    \centering
    \subfloat[$K_4^*$-graphs.]{\label{fig:K4}\includegraphics[width=0.8\textwidth]{figures/rem_cases_d.png}}
    \end{minipage}%
    \hfill
    \begin{minipage}{.48\linewidth}
    \centering
    \subfloat[$K_{3,3}^*$-graphs.]{\label{fig:K33}\includegraphics[width=0.8\textwidth]{figures/rem_cases_e.png}}
    \end{minipage}%

    \caption{The remaining possibilities for types of graphs in the set $\mathcal G$. Here, the lines between vertices in the graphs correspond to paths, possibly with several inner vertices. The hatch mark on one of the paths in (c) indicates that one of the six paths between the vertices of degree $6$ necessarily must be an edge, or it would immediately be seen to not be in $\mathcal G$. If we interpret the paths of the two graphs in the bottom row as edges, observe that the resulting graphs would be isomorphic to $K_4$ and $K_{3,3}$.}
    \label{fig:rem_cases}
\end{figure}

\begin{proof}
As discussed above, no graph $G \in \mathcal G$ has a proper barbell subgraph. Recall that if $G$ is not a cycle graph, it must contain at least two cycle subgraphs: as such, any such $G$ that is not a cycle, barbell, or a $\theta$-graph must contain a proper $\theta$-subgraph by Lemma \ref{two_cycle_subgphs}. Take any such graph $G$ with a proper $\theta$-subgraph $G_\theta$, and let its two vertices of degree $3$ be denoted $v$ and $w$. 

\paragraph{Case 1: Either $v$ or $w$ has degree greater than $3$.}

Without loss of generality, say $v$ has degree greater than $3$, and consider an edge incident to $v$ that does not lie in $G_\theta$. Continue selecting edges along a path until we return to a vertex in $G_\theta$: this must be $w$ to avoid a proper barbell subgraph of $G$; call the resulting graph $G_\theta'$. From here, it is straightforward to observe that all vertices in $V(G_\theta) \setminus \{v, w\}$ have degree $2$ in $G$, and the only way we can append to $G_\theta'$ is to include additional paths between $v$ and $w$.

First assume all paths from $v$ to $w$ in $G$ are not edges: any optimal traversal of $G$ (i.e. that yielding $g(\FS(G, \Star_n))$ that uses a path traverses it at least twice, so $G$ has at most five paths (indeed, if $G$ had six or more paths, invoke the trajectory given by Proposition \ref{theta_bd} on the three shortest paths). If a path from $v$ to $w$ has (exactly) one edge, then by analogous reasoning, $G$ can have at most six paths.

\paragraph{Case 2: Both $v$ and $w$ have degree $3$.}

Let (without loss of generality) $x_1 \in V(\mathcal C_1) \setminus \{v, w\}$: selecting edges along a path until we return to a vertex in $G_\theta$ must yield $x_1' \in V(\mathcal C_2) \setminus \{v, w\}$. If there exists a second path from $x_2 \in V(\mathcal C_1) \setminus \{v, w\}$ to $x_2' \in V(\mathcal C_2) \setminus \{v, w\}$, we must have $d_{\mathcal C_1}(x_1, v) < d_{\mathcal C_1}(x_1', v)$ and $d_{\mathcal C_2}(x_2, w) < d_{\mathcal C_2}(x_2', w)$ (again without loss of generality): if $d_{\mathcal C_1}(x_1, v) < d_{\mathcal C_1}(x_1', v)$ and $d_{\mathcal C_2}(x_2, w) > d_{\mathcal C_2}(x_2', w)$, there would exist a proper barbell subgraph, as depicted in Figure \ref{fig:theta_cross_paths}(a). However, three such paths necessarily yield a proper barbell subgraph: this follows from the preceding discussion if we do not have $d_{\mathcal C_1}(x_1, v) < d_{\mathcal C_1}(x_2, v) < d_{\mathcal C_1}(x_3, v)$, $d_{\mathcal C_2}(x_1', w) < d_{\mathcal C_2}(x_2', w)$, $d_{\mathcal C_2}(x_3', w))$, while this case still yields a proper barbell subgraph as depicted in Figure \ref{fig:theta_cross_paths}(b).
\end{proof}

\begin{figure}[ht]
    \centering
    \subfloat[Two paths above $G_\theta$ that fail to cross.]{\includegraphics[width=0.4\textwidth]{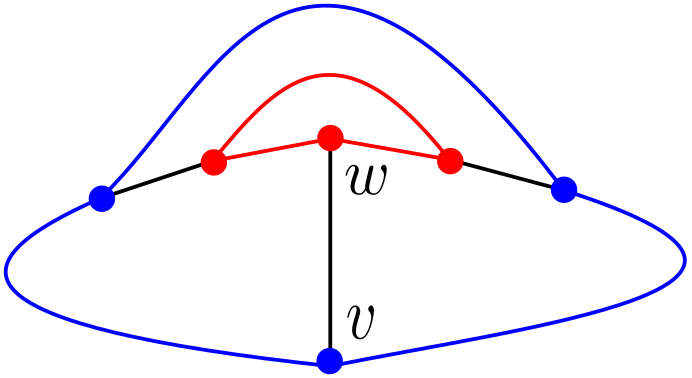}\label{fig:demo_X3}}
    \hfill
    \subfloat[Three paths above $G_\theta$ with any two crossing.]{\includegraphics[width=0.4\textwidth]{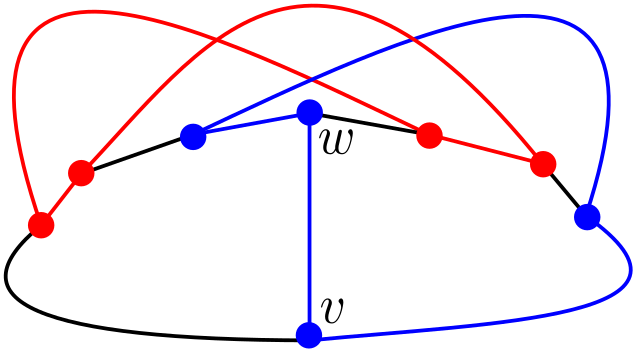}\label{fig:demo_Y3}}
    \caption{Illustrations corresponding to Case 2 in the proof of Proposition \ref{remaining_cases}, showing the existence of a proper barbell subgraph. The two cycle subgraphs of the proper barbell subgraph are colored in red and blue. The vertices of degree $3$ in the proper $\theta$-subgraph $G_\theta$ are denoted by $v$ and $w$.}
    \label{fig:theta_cross_paths}
\end{figure}

\begin{proposition} \label{fourth_traj}
Let $X$ be a $\theta_4$-graph with paths $\mathcal P_1, \mathcal P_2, \mathcal P_3, \mathcal P_4$ between the vertices $v, w \in V(X)$ with degree $4$, with corresponding path lengths $p_1 = 1 < p_2 \leq p_3 \leq p_4$, respectively. If $4+4p_2+2(p_3+p_4) < \min\{p_2(p_2+1), 4(p_2+p_3+p_4), 6(1+p_2+p_3)\}$, then $X \in \mathcal G$.
\end{proposition}

\noindent Such graphs $X$ certainly exist. For example, let $(p_1, p_2, p_3, p_4) = (1, p, p, p)$ for $p \geq 8$: we have that $4+4p_2+2(p_3+p_4) < \min\{p_2(p_2+1), 4(p_2+p_3+p_4), 6(1+p_2+p_3)\}$.

\begin{proof}
Let the vertices of degree $3$ be denoted $v, w$, and label the four paths from $v$ to $w$ in $X$ by $\mathcal P_1 = \{v_1 = v, \dots, v_{m_1} = w\}$, $\mathcal P_2 = \{w_1 = v, \dots, w_{m_2} = w\}$, $\mathcal P_3 = \{x_1 = v, \dots, x_{m_3} = w\}$, $\mathcal P_4 = \{y_1 = v, \dots, y_{m_4} = w\}$ in order following the path. Consider the following sequences of swaps.
\begin{align*}
    & \mathcal S_1 = (v \ v_2)(v_2 \ v_3)\dots(v_{m_1-1} \ w) \\
    & \mathcal S_2 = (v \ w_2)(w_2 \ w_3) \dots (w_{m_2-1} \ w) \\
    & \mathcal S_3 = (v \ x_2)(x_2 \ x_3) \dots (x_{m_3-1} \ w) \\
    & \mathcal S_4 = (v \ y_2)(y_2 \ y_3) \dots (y_{m_4-1} \ w)
\end{align*}
Starting $\mathfrak n$ at $v$, consider the following sequence $\mathcal S$ of $(X, \Star_n)$-friendly swaps. As in the proof of Proposition \ref{theta_bd}, $\mathcal S$ can conveniently be represented by a (circular) word $\sigma_w$ on a $4$-letter alphabet $\{1, 2, 3, 4\}$, where $i \in [4]$ corresponds to applying $\mathcal S_i$. We elect to use this representation of $\mathcal S$, which we denote $\mathcal S_\mathcal W$.
\begin{align*}
    \mathcal S_{\mathcal W}: 312412132142
\end{align*}
It is straightforward to observe that performing $\mathcal S_{\mathcal W}$ yields a cycle in $\FS(X, \Star_n)$ with size $4+4p_2+2(p_3+p_4)$. (See Figure \ref{fig:fourth_traj} for an illustration of this trajectory.) The statement follows from comparison with the optimal trajectories given by Corollary \ref{quad_bd} and Proposition \ref{theta_bd}.
\end{proof}

\begin{figure}[ht]
    \centering
    \includegraphics[width=0.4\textwidth]{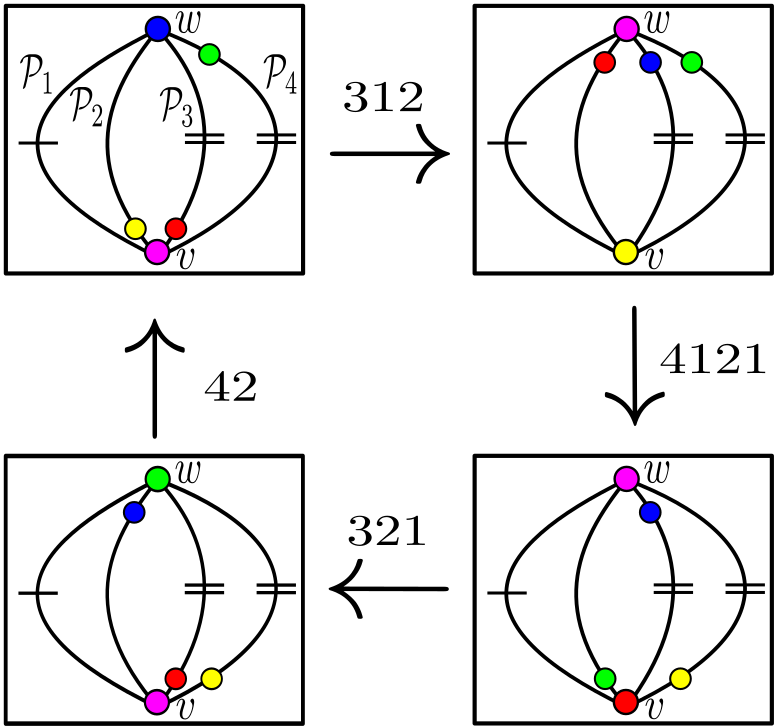}
    \caption{Illustration of the trajectory $\mathcal S_\mathcal W$ given above; the paths $\mathcal P_1$ to $\mathcal P_4$ are from left to right. We color in the preimages of some key vertices in $V(\Star_n)$ (note that these are, or are adjacent to, the vertices $v$ and $w$ in the first configuration), and in particular note that the preimage of $\mathfrak n$ is colored pink. The hatch mark above $\mathcal P_1$ indicates that it is an edge, while the two hatch marks above $\mathcal P_3$ and $\mathcal P_4$ indicate that both of these paths are traversed exactly twice over the course of $\mathcal S_\mathcal W$.}
    \label{fig:fourth_traj}
\end{figure}

\noindent In particular, this result shows that $\mathcal G \neq \tilde{\mathcal G}$, so the inclusion of Theorem \ref{girth_subset} is strict. Thus, we can enhance $\tilde{\mathcal G}$ to also include all instances of $\theta_4$-graphs studied by Proposition \ref{fourth_traj}, and conclude that we must have $\tilde{\mathcal G} \subset \mathcal G$; henceforth, $\tilde{\mathcal G}$ will refer to this enhanced set of simple graphs (in particular, that described by Theorem \ref{girth_subset_intro}). Also observe that this result yields that no instances of $\tilde{\theta_6}$-graphs in Figure \ref{fig:rem_cases}(c) can be in $\mathcal G$. Specifically, consider any such graph $X$; let $\phi = (p_1=1, \dots, p_6)$ be the $6$-vector containing the lengths of the paths $\mathcal P_1, \dots, \mathcal P_6$ between the vertices of degree $6$ in increasing order, and $\eta = (n_1, \dots, n_6)$ be such that for $i \in [6]$, $n_i$ denotes the number of times $\mathcal P_i$ is traversed by $\mathfrak n$. Assume $\mathfrak n$ traverses all six paths in a cycle in $\FS(X, \Star_n)$, so $n_i \geq 1$ for all $i \in [6]$. Observe that $\phi \cdot \eta$ represents the size of this cycle subgraph in $\FS(X, \Star_n)$, and that $n_i \geq 2$ for all $2 \leq i \leq 6$ since all vertices of $V(\Star_n)$ originally upon inner vertices of such paths must return to their original positions. Such a trajectory clearly cannot improve that which lies strictly on the paths $\mathcal P_1, \dots, \mathcal P_4$ given by Proposition \ref{fourth_traj}, which has $\eta' = (4, 4, 2, 2, 0, 0)$ (i.e. $\phi \cdot \eta > \phi \cdot \eta'$).

Therefore, we have that $\tilde{\mathcal G} \subset \mathcal G \subset \mathcal G'$, and that $\mathcal G \setminus \tilde{\mathcal G}$ consists strictly of instances of graphs in Figure \ref{fig:final_rem_cases}. We conjecture that the set of graphs $\tilde{\mathcal G}$ described above is all of $\mathcal G$.

\begin{conjecture}
$\mathcal G = \tilde{\mathcal G}$.
\end{conjecture}

\begin{figure}[ht]
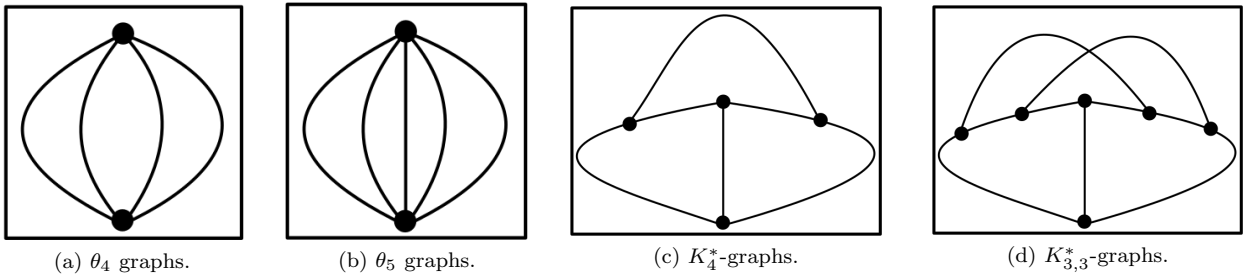

    \centering
    \begin{minipage}{.195\linewidth}
    \centering
    \subfloat[$\theta_4$ graphs.]{\label{fig:theta4_final}\includegraphics[width=\textwidth]{figures/rem_cases_a.png}}
    \end{minipage}%
    \hfill
    \begin{minipage}{.195\linewidth}
    \centering
    \subfloat[$\theta_5$ graphs.]{\label{fig:theta5_final}\includegraphics[width=\textwidth]{figures/rem_cases_b.png}}
    \end{minipage}%
    \hfill
    \begin{minipage}{.25\linewidth}
    \centering
    \subfloat[$K_4^*$-graphs.]{\label{fig:K4*_final}\includegraphics[width=\textwidth]{figures/rem_cases_d.png}}
    \end{minipage}
    \hfill
    \begin{minipage}{.25\linewidth}
    \centering
    \subfloat[$K_{3,3}^*$-graphs.]{\label{fig:K33*_final}\includegraphics[width=\textwidth]{figures/rem_cases_e.png}}
    \end{minipage}%
    \caption{The remaining possibilities for graphs in the set $\mathcal G$ that remain unresolved by the present work.}
    \label{fig:final_rem_cases}
\end{figure}

\section{Acknowledgements}

This research was conducted at the University of Minnesota Duluth REU and was supported, in part, by NSF-DMS grant 1949884 and NSA Grant H98230-20-1-0009. We would like to thank Professor Joseph Gallian for organizing the Duluth REU, and are deeply grateful to Colin Defant and Noah Kravitz (authors of the papers that introduced friends-and-strangers graphs) for many helpful conversations. We in particular thank Colin Defant for suggesting induction on the number of ears of a biconnected graph as a possible approach toward proving Theorem \ref{conjec_7.1}, introducing us to Inkscape as a tool for making figures, and for many productive comments on a draft of the manuscript, and Noah Kravitz for several fruitful discussions on both of the problems studied in this work. We would also like to sincerely thank David Rolnick for helpful conversations and for providing access to faster compute.

\section{References}

\printbibliography[heading=none]
\end{document}